\documentclass[12pt]{article}
\usepackage{amssymb,amsfonts,amsmath,amsthm,amsrefs,dsfont,hyperref}
\usepackage[letterpaper, portrait, margin=29mm]{geometry}

\newcommand{\1}{\mathds{1}}
\newcommand{\0}{\mathds{O}}

\newcommand{\Z}{\mathbb{Z}}
\newcommand{\Q}{\mathbb{Q}}
\newcommand{\R}{\mathbb{R}}

\newcommand{\N}{\mathbb{N}}

\newcommand{\Bbo}{\overline{\mathrm{B}}}
\newcommand{\8}{\infty}

\newcommand{\Ker}{\mathrm{Ker~}}

\newcommand{\Co}{\mathcal{C}}
\newcommand{\Fo}{\mathcal{F}}

\newcommand{\Lo}{\mathcal{L}}

\newcommand{\Int}{\mathrm{int}}

\newcounter{dummy} \numberwithin{dummy}{section}
\newtheorem{theorem}[dummy]{Theorem}
\newtheorem{lemma}[dummy]{Lemma}

\newtheorem{proposition}[dummy]{Proposition}
\newtheorem{corollary}[dummy]{Corollary}
\newtheorem{question}[dummy]{Question}
\theoremstyle{remark}
\newtheorem{remark}[dummy]{Remark}
\newtheorem{example}[dummy]{Example}

\begin{document}

\title{Order continuity and regularity on vector lattices and on lattices of continuous functions}
\author{Eugene Bilokopytov\footnote{Email address bilokopy@ualberta.ca, erz888@gmail.com.}}
\maketitle

\begin{abstract}
We give several characterizations of order continuous vector lattice homomorphisms between Archimedean vector lattices. We reduce the proofs of some of the equivalences to the case of composition operators between vector lattices of continuous functions, and we obtain a characterization of order continuity of such operators. Motivated by this, we investigate various properties of the sublattices of the space $\Co\left(X\right)$, where $X$ is a Tychonoff topological space. We also obtain several characterizations of a regular sublattice of a vector lattice, and show that the closure of a regular sublattice of a Banach lattice is also regular.\medskip


MSC2020 46A40, 46B42, 46E25, 47B33, 47B60, 54C10.
\end{abstract}

\section{Introduction}

One of the important properties of linear operators between vector lattices is order continuity, i.e. continuity with respect to the order convergence. In this article we characterize vector lattice homomorphisms with this property. A concept which is closely related to the order continuity of homomorphisms is regularity of sublattices, i.e. the property that supremum of an arbitrary set agrees with the supremum calculated with respect to the ambient lattice. Consequently, we obtain several characterizations of such sublattices.

The question of when a homomorphism is order continuous seems to have not been thoroughly investigated. There is a characterization of surjective order continuous homomorphisms (see \cite[Theorem 18.13]{zl}), as well as a characterization of vector lattices on which all homomorphisms are order continuous (see \cite{fremlin,tucker}). We adjust the former characterization to the case when the  homomorphism is not necessarily a surjection, and add some new equivalent conditions (see Theorem \ref{hoc}). Our proof of some of the equivalences reduces the abstract case to the case of a composition operator between lattices of continuous functions. This motivates us to characterize order continuity of a composition operator (see Theorem \ref{com}). We take this as an opportunity to study such lattices in their own right. In fact, we add to the growing literature that deals with the vector lattice properties of the spaces $\Co\left(X\right)$, where $X$ is not necessarily compact (see e.g. \cite{djvdw,hager,kv2,kv1,vdw}). Parallel to the study of composition operators we also consider sublattices of $\Co\left(X\right)$, and in particular characterize their regularity.

In Section \ref{p} we recall the basic concepts from the vector lattice theory, in particular various notions of subspaces of a vector lattice. Section \ref{d} is devoted to order continuity of a vector lattice homomorphism. We show that a continuous homomorphism between Banach lattices is order continuous provided that its restriction to a dense regular sublattice is order continuous (Theorem \ref{choc}). In this section we also state and partially prove our main theorem. As a consequence we obtain characterizations of regular sublattices (Corollary \ref{reg}) and show that a closure of a regular sublattice of a Banach lattice is regular (Corollary \ref{nreg}).

In Section \ref{x} we gather some standard facts about $\Co\left(X\right)$ and present them in the form adapted for the needs of this investigation. Section \ref{s} is dedicated to sublattices of $\Co\left(X\right)$, where in particular we show that if $X$ is locally compact, then a closure of a regular sublattice of $\Co\left(X\right)$ is regular (Theorem \ref{sublo}). In Section \ref{m} we discuss various classes of continuous maps that will be essential in Section \ref{c}, which is dedicated to order continuity of the composition operators. More specifically, Theorem \ref{com} relates some vector lattice properties of a composition operator (including order continuity) with topological properties of its symbol. This enables us to complete the proof of Theorem \ref{hoc}, as well as produce examples (examples \ref{inte}, and \ref{intero}) of closed regular (order dense) sublattices of a Banach lattice whose intersection is not regular (order dense).

\section{Preliminaries}\label{p}

We start with recalling some concepts from the vector lattice theory. Let $F$ be a vector lattice. Two elements $e,f\in F$ are called \emph{disjoint} (denoted by $e\bot f$), if $\left|e\right|\wedge \left|f\right|=0_{F}$. If $G\subset F$, then its \emph{disjoint complement} $G^{d}=\left\{f\in F,~ \forall g\in G:~ f\bot g\right\}$. It is clear that $G\cap G^{d}\subset \left\{0_{F}\right\}$. A net $\left\{f_{\gamma}\right\}_{\gamma\in \Gamma}$ \emph{order converges} to $f\in F$ (we denote it by $f_{\gamma}\xrightarrow[\gamma\in \Gamma]{o}f$) if there is a net $\left\{g_{\beta}\right\}_{\beta\in B }$ such that $g_{\beta}\downarrow 0_{F}$ (this means that $\left\{g_{\beta}\right\}_{\beta\in B}$ is decreasing and its infimum is $0_{F}$) and for every $\beta\in B$ there is $\gamma\in \Gamma$ such that $\left|f-f_{\alpha}\right|\le g_{\beta}$, as soon as $\alpha\ge \gamma$. It is easy to see that an increasing net order converges to its supremum once the latter exists. Additionally, we say that $\left\{f_{\gamma}\right\}_{\gamma\in \Gamma}$ \emph{unboundedly order (uo) converges} to $f\in F$ ($f_{\gamma}\xrightarrow[\gamma\in \Gamma]{uo}f$) if $\left|f-f_{\gamma}\right|\wedge h \xrightarrow[\gamma\in \Gamma]{o}0_{F}$, for every $h\in F_{+}$.\medskip

Let us also recall the menagerie of the types of subspaces of a vector lattice. A linear subspace $E\subset F$ is called a \emph{vector sublattice} if it is closed with respect to the lattice operations of $F$ (equivalently, if $f\in E$ implies $\left|f\right|\in E$). In the sequel we will drop the word ``vector'' and call $E$ simply a sublattice. For any finite $G\subset E$ we have $\sup_{E}G=\sup_{F}G$ -- the equality of the supremums with respect to the order structure of $E$ and $F$. If the same equality holds for any $G\subset E$ whose supremum in $E$ exists, we will call $E$ a \emph{regular sublattice}. In fact, this equality is enough to check for monotone nets (or even nets decreasing to $0_{F}$). If for every $f>0_{F}$ there is $e\in E$ such that $0_{F}< e\le f$, then $E$ is called an \emph{order dense sublattice} of $F$. It is easy to see that every order dense sublattice is regular. Another subclass of regular sublattices is formed by \emph{ideals}, i.e. sublattices $E\subset F$ such that if $0_{F}\le f\le e\in E$, then $f\in E$. We will call a subset $G\subset F$ \emph{majorizing} if there is no proper ideal of $F$ that contains $G$. Equivalently, for every $f\in F_{+}$ there is a positive linear combination of absolute values of elements of $G$ that is larger than $f$.

An ideal is called a \emph{band}, if it is closed with respect to the order convergence (equivalently, closed with respect of taking supremums of arbitrary sets; this is enough to check only for increasing nets). One can show that if $F$ ia Archimedean and $G\subset F$, then $G^{dd}$ is the minimal band in $F$ that contains $G$. A band $E$ is called a \emph{projection band} if $F=E+E^{d}$. This condition is equivalent to existence of $P\in\Lo\left(F\right)$ such that $0_{\Lo\left(F\right)}\le P=P^{2}\le Id_{F}$ and $PF=E$. Here $\Lo\left(F\right)$ is the linear space of linear operators on $F$ ordered by the relation $T\le S$ if $Tf\le Sf$, for all $f\in F_{+}$. We say that $F$ has the \emph{projection property (PP)} if every band in $F$ is a projection band, and the \emph{principal projection property (PPP)} if every element of $F$ generates a projection band, i.e. $F=\left\{f\right\}^{d}+\left\{f\right\}^{dd}$, for every $f\in F$.\smallskip

\begin{remark}\label{tran}
We will often use the following observation. Let $\mathcal{S}\left(F\right)$ denote the collection of sublattices / regular sublattices / order dense sublattices / ideals / bands / projection bands of $F$, and let $H\subset E\subset F$ be linear subspaces. Then, $H\in \mathcal{S}\left(F\right)\Rightarrow H\in \mathcal{S}\left(E\right)$, and the converse is true if additionally $E\in \mathcal{S}\left(F\right)$.
\qed\end{remark}\smallskip

\begin{remark}\label{menag}
If $E$ is an ideal of $F$ and $H$ is a sublattice of $F$, then $E\cap H$ is an ideal in $H$. If $H$ is order dense, then $E\cap H$ is order dense in $E$. Moreover, if $H$ is an order dense sublattice of $E$, then $H+E^{d}$ is order dense in $F$ and such that $\left(H+E^{d}\right)\cap E=H$. If $G$ is a projection band, then $E\cap G$ is a projection band in $E$. Indeed, if $P$ is the corresponding projection, then $PE= E\cap G$, as $0_{F}\le Pe\le e\in E$, for any $e\in E_{+}$, and $Pg=g$, for any $g\in  E\cap G$; moreover, $\left.P\right|_{E}$ satisfies $0_{\Lo\left(E\right)}\le \left.P\right|_{E}=\left.P\right|_{E}^{2}\le Id_{E}$, and so it is a projection in $E$ whose range is $E\cap G$.
\qed\end{remark}\smallskip

Similarly to how a supremum of an infinite subset of a sublattice $E$ of $F$ can be different when calculating with respect to $F$ and $E$, the disjoint complement may depend on the ambient lattice as well. For $G\subset E$ let $G^{d}_{F},G^{d}_{E},G^{dd}_{F},G^{dd}_{E}$ stand for the disjoint complement, or the second complement with respect to $F$ or $E$.

\begin{proposition}\label{dis}Let $E$ be a sublattice of $F$ and let $G\subset E$. Then $G^{d}_{E}=G^{d}_{F}\cap E$ and $G^{dd}_{F}\cap E\subset G^{dd}_{E}$. If additionally $G^{dd}_{F}\cap E$ is a band in $E$, then $G^{dd}_{F}\cap E= G^{dd}_{E}$. If $G$ is itself a band in $E$, then $G=G^{dd}_{F}\cap E$.
\end{proposition}
\begin{proof}
The first main claim is trivial since $f\bot_{E} g\Leftrightarrow f\bot_{F} g$. For the second claim from $G^{d}_{E}\subset G^{d}_{F}$ we get $G^{dd}_{E}=\left(G^{d}_{E}\right)^{d}_{F}\cap E\supset \left(G^{d}_{F}\right)^{d}_{F}\cap E=G^{dd}_{F}\cap E$. If $G^{dd}_{F}\cap E$ is a band in $E$, then it is a band that contains $G$, and so it contains $G^{dd}_{E}$. If $G$ is a band in $E$, then  $G^{dd}_{E}=G\subset G^{dd}_{F}\cap E\subset G^{dd}_{E}$.
\end{proof}

Let us recall few more definitions from the vector lattice theory. $F$ is called \emph{Archimedean} if $\frac{1}{n}f\downarrow 0_{F}$, for every $f\in F_{+}$. This condition is equivalent to the fact that if $ne\le f$, for $e\in F_{+}$ and every $n\in\N$, then $e= 0_{F}$. Consequently, a sublattice of an Archimedean lattice is Archimedean.

A \emph{principal ideal} is an ideal of the form $F_{e}=\bigcup\limits_{\alpha\ge 0}\alpha\left[-e,e\right]$ -- the minimal ideal that contains $e\in F_{+}$. By default this ideal is endowed with the lattice norm $\|\cdot\|_{e}$ defined by $\|f\|_{e}=\inf\left\{\alpha\ge 0:~ \left|f\right|\le\alpha e\right\}$. If $F=F_{e}$, then $e$ is called a \emph{strong unit} and $F$ is called \emph{unital}. $F$ has the $\sigma$\emph{-property} if every sequence is contained in a principal ideal. One can show that every Banach lattice has the $\sigma$-property (see also \cite{hager,hvm}).

$F$ has a \emph{countable supremum property} if for every $G\subset F$ and $g\in F$ such that $g=\sup G$ there is a sequence $\left\{g_{n}\right\}_{n\in\N}\subset G$ such that $\sup\left\{g_{n}\right\}_{n\in\N}=g$. This property is often satisfied by the Kothe function spaces (see \cite[Lemma 2.6.1]{mn}), and the spaces of continuous functions (see \cite{kv2}).\medskip

A subset $A\subset F$ is called \emph{solid}, if together with every $e\in A$ it contains $\left[-\left|e\right|,\left|e\right|\right]$. A (Hausdorff) vector topology on $F$ is called \emph{locally solid} if it has a base at $0_{F}$ that consists of solid sets. An example of such topology is produced by a \emph{lattice semi-norm}, i.e. a seminorm $\|\cdot\|$ on $F$ such that $\|e\|\le \|f\|$, for every $e,f\in F$ with $\left|e\right|\le \left|f\right|$. It is not hard to show that $F_{+}$ is closed, and if $E$ is a sublattice of $F$, then $\overline{E}$ is also a sublattice with $\overline{E_{+}}=\overline{E}_{+}$.

\begin{proposition}\label{closure}
Let $F$ be a vector lattice with a locally solid topology. Then:
\item[(i)] If $E$ is an ideal in $F$, then for every $e\in \overline{E}_{+}$ we have $e\in \overline{\left[0_{F},e\right]\cap E}$. In particular, $E$ is order dense in $\overline{E}$.
\item[(ii)] If $F$ is a Banach lattice, and $E$ is regular and dense sublattice of $F$, then $E$ is order dense.
\end{proposition}
\begin{proof}
(i): Let $U$ be an open solid neighborhood of $0_{F}$. Since $e\in \overline{E}$ there is $f\in E$ such that $\left|e-f\right|\in U$. Let $g=e\wedge f^{+}\in \left[0_{F},e\right]$. Since $E$ is an ideal, $g\in E$; moreover $0_{F}\le e-g\le \left|e-f\right|$. Hence, since $U$ is solid, we have $\left|e-g\right|\in U$, and since $U$ was chosen arbitrarily, the result follows.\medskip

(ii): Let $f\in F_{+}$ be such that $\|f\|=2$. For every $n\in\N$ there is $e_{n}\in E_{+}$ such that $\|f-e_{n}\|\le \frac{1}{2^{n}}$. Let $e=\sum\limits_{n\in\N}\left|f-e_{n}\right|$. Then, $\|e\|\le 1$, and for every $n\in\N$ we have $e_{n}\ge f-\left|f-e_{n}\right|\ge f- e$, and so $e_{n}\ge\left(f-e\right)^{+}>0_{F}$, where the latter follows from $f>0_{F}$ and $\|e\|<\|f\|$. Hence, $\inf\limits_{n\in\N} e_{n}\ne 0_{F}$ (it either does not exist or is greater than $0_{F}$), and since $E$ is regular, there is $g\in E$ such that $e_{n}\ge g>0_{F}$, for every $n\in\N$. Since $e_{n}\xrightarrow[n\to\8]{}f$, it follows that $f\ge g$.
\end{proof}

An example of a non-regular dense sublattice of a Banach lattice see in Example \ref{prrd}. Also note that intersection of two closed regular (order dense) sublattices of a Banach space does not have to be regular (order dense), as will be shown by examples \ref{inte}, and \ref{intero}.

If $E$ is a vector lattice, then a linear operator $T:F\to E$ between vector lattices is called a \emph{vector lattice homomorphism} if it preserves the lattice operations (it is enough to check that it preserves the absolute value). In the sequel we will drop ``vector lattice'' from the term ``vector lattice  homomorphism'', as we will not deal with homomorphisms of other structures. It is easy to see that if $T$ is a homomorphism, then $\Ker T$ is an ideal, $TF$ is a sublattice of $E$, and $T\ge 0_{\Lo\left(F,E\right)}$. An \emph{isomorphism} is a bijective homomorphism.

\begin{remark}\label{homo}
A useful property of homomorphisms is that if $Tf\le Tg\le Th$, and $f\le h$, then there is $e\in \left[f,h\right]$ such that $Te=Tg$. In other words, $T\left[f,h\right]=\left[Tf,Th\right]\cap TF$. Indeed, if $e=\left(g\vee f\right)\wedge h\in \left[f,h\right]$, we have $Te=\left(Tg\vee Tf\right)\wedge Th=Tg$.
\qed\end{remark}

\section{Order continuity of positive operators and homomorphisms}\label{d}

In this section we examine the connections between two notions of continuity for a positive operator -- topological continuity and order continuity. A special attention will be paid to the order continuous homomorphisms. Recall that a linear operator $T:F\to E$ is called \emph{order continuous} if it preserves order convergence. One can show that for positive operators order continuity is equivalent to the property that $Tf_{\gamma}\downarrow 0_{E}$, whenever $f_{\gamma}\downarrow 0_{F}$.

It is easy to see that the composition of order continuous operators is order continuous, and that every isomorphism is an order continuous operator. Also, a sublattice $H$ of $F$ is regular if and only if its inclusion into $F$ is an order continuous homomorphism. Consequently, a restriction of an order continuous operator to a regular sublattice is order continuous as a composition of two order continuous maps (the other one being the inclusion operator). Consider the set of order continuity of an operator (see also \cite[Theorem 88.6]{zaanen}).

\begin{proposition}\label{lo}
Let $T:F\to E$ be a positive operator between vector lattices and let $H=\left\{h\in F:~ \left.T\right|_{F_{\left|h\right|}}\mbox{ is order continuous}\right\}$. Then, $H$ is an ideal, and $\left.T\right|_{H}$ is order continuous. Moreover, if $F$ and $E$ are endowed with locally solid topologies, and $T$ is continuous, then $H$ is closed.
\end{proposition}
\begin{proof}
Clearly, $\alpha H\subset H$, for every $\alpha>0$. If $h\in H$ and $\left|g\right|\le \left|h\right|$, then $F_{\left|g\right|}$ is a regular sublattice of $F_{\left|h\right|}$, and so order continuity of $\left.T\right|_{F_{\left|h\right|}}$ implies order continuity of $\left.T\right|_{F_{\left|g\right|}}$. Hence, $H$ is a solid set. Let $g,h\in H_{+}$ and $g+h\ge f_{\gamma}\downarrow 0_{F}$. Then, $g\ge f_{\gamma}\wedge g\downarrow 0_{F}$ and $h\ge \left(f_{\gamma}-g\right)^{+}\downarrow 0_{F}$, and since $\left.T\right|_{F_{h}}$ and $\left.T\right|_{F_{g}}$ are order continuous,  $Tf_{\gamma}=T\left(f_{\gamma}\wedge g\right)+T\left(f_{\gamma}-g\right)^{+}\downarrow 0_{E}$. Hence, $g+h\in H$, and so $H$ is an ideal.\medskip

Let $H\ni f_{\gamma}\downarrow 0_{F}$. Fix an index $\beta$. We have $f_{\beta}\in H$ and $f_{\beta}\ge f_{\gamma}\downarrow 0_{F}$ (for $\gamma\ge \beta$). Since $\left.T\right|_{F_{f_{\beta}}}$ is order continuous, it follows that $Tf_{\gamma}\downarrow 0_{E}$. Hence, $\left.T\right|_{H}$ is order continuous.\medskip

Let us prove the last claim. Assume that $h\in \overline{H}_{+}$, $h\ge f_{\gamma}\downarrow 0_{F}$ and $0_{E}\le e\le Tf_{\gamma}$, for every $\gamma$. Let $U$ be a solid open neighborhood of $0_{E}$. Since $T$ is continuous there is an open solid neighborhood $V$ of $0_{F}$ such that $TV\subset U$. Since $h\in \overline{H}_{+}$, from part (i) of Proposition \ref{closure}, there is $g\in \left[0_{F},h\right]\cap H\cap \left(h+V\right)$. Then, $g\ge f_{\gamma}\wedge g\downarrow 0_{F}$, and $f_{\gamma}-f_{\gamma}\wedge g=\left(f_{\gamma}-g\right)^{+}\le h-g$, for every $\gamma$. Then, $$e\le Tf_{\gamma}= T\left(f_{\gamma}\wedge g\right)+T\left(f_{\gamma}-f_{\gamma}\wedge g\right)\le T\left(f_{\gamma}\wedge g\right)+T\left(h- g\right),$$ for every $\gamma$, and since $g\in H$, it follows that $e\le T\left(h-g\right)\in U$. Since $U$ was chosen arbitrarily, we conclude that $e=0_{E}$, and so $Tf_{\gamma}\downarrow 0_{E}$, from where $h\in H$.
\end{proof}

\begin{corollary}\label{oc}For a positive operator $T:F\to E$ between Archimedean vector lattices the following conditions are equivalent:
\item[(i)] $T$ is order continuous;
\item[(ii)] $\left.T\right|_{H}$ is order continuous, for every regular sublattice $H$ of $F$;
\item[(iii)] There is a majorizing $G\subset F_{+}$ such that $\left.T\right|_{F_{g}}$ is order continuous, for every $g\in G$;
\item[(iv)] There is an order dense and majorizing sublattice $H$ such that $\left.T\right|_{H}$ is order continuous.
\end{corollary}
\begin{proof}
First, note that (i)$\Rightarrow$(ii) was discussed in the beginning of the section, (ii)$\Rightarrow$(iii) and (ii)$\Rightarrow$(iv) are trivial, and (iii)$\Rightarrow$(i) follows immediately from Proposition \ref{lo}.\medskip

(iv)$\Rightarrow$(i): Let $f_{\gamma}\downarrow 0_{F}$ and let $e\in E$ be such that $0_{E}\le e\le Tf_{\gamma}$, for every $\gamma$. Let $G=\left\{g\in H~ \exists \gamma: f_{\gamma}\le g\right\}$, which is directed downwards. Note that $e\le TG$, since for every $g\in G$ there is $\gamma$ such that $e\le Tf_{\gamma}\le Tg$.

Let $f\in H$ be such that $0_{F}\le f\le G$. Fix $\gamma$. Since $H$ is majorizing, there is $h\in H$ with $h\ge f_{\gamma}$, while since $H$ is order dense, $\sup\left[0_{F},h-f_{\gamma}\right]\cap H=h-f_{\gamma}$ (see \cite[Theorem 1.34]{ab}). Hence, $\inf\left[f_{\gamma},h\right]\cap H=h-\sup\left[0_{F},h-f_{\gamma}\right]\cap H=f_{\gamma}$, and since $f\le G\supset \left[f_{\gamma},h\right]\cap H$, it follows that $f\le f_{\gamma}$. Since $\inf\limits_{\gamma\in \Gamma} f_{\gamma}=0_{F}$, we conclude that $f=0_{F}$.

Hence, $\inf_{H}G=0_{F}$, and since $G$ is a decreasing net in $H$ and $\left.T\right|_{H}$  is order continuous, $\inf_{E}TG=0_{E}$. Therefore, $e=0_{E}$, and so $Tf_{\gamma}\downarrow 0_{E}$. Thus, $T$ is order continuous.
\end{proof}

Let $F$ be a locally solid lattice. We will call a subset $G\subset F$ \emph{almost majorizing} if there is no proper closed ideal of $F$ that contains $G$. Obviously, a dense subset of $F_{+}$ is almost majorizing. It follows from Proposition \ref{lo} that if in the setting of the proposition $T$ was continuous with respect to some locally solid topologies on $E$ and $F$, and there was an almost majorizing $G\subset F_{+}$ such that  $\left.T\right|_{F_{g}}$ is order continuous, for every $g\in G$, then $T$ would be order continuous.

\begin{theorem}\label{choc}
Let $T$ be a continuous operator between locally solid lattices $F$ and $E$. Let $H$ be a sublattice of $F$ such that $\left.T\right|_{H}$ is positive and order continuous. Then, $T$ is positive and order continuous provided that one of the following conditions is satisfied:
\item[(i)] $H$ is a dense ideal in $F$;
\item[(ii)] $H$ is order dense and almost majorizing;
\item[(iii)] $F$ is a Banach lattice and $H$ is a dense regular sublattice of $F$;
\item[(iv)] $F$ is a Banach lattice with the countable supremum property and $H$ is dense.
\end{theorem}
\begin{proof}
Positivity of $T$ follows from  $TF_{+}=T\overline{H}_{+}=T\overline{H_{+}}\subset \overline{TH_{+}}\subset \overline{E_{+}}=E_{+}$.\medskip

If $H$ is a dense ideal, $F_{h}$ is a regular sublattice of $H$ for every $h\in H_{+}$, and so $\left.T\right|_{F_{h}}$ is order continuous. Since $H_{+}$ is almost majorizing, it follows from the comment before the theorem that $T$ is order continuous.

If $H$ is order dense and almost majorizing, let $G$ be the (dense) ideal generated by $H$, i.e. an ideal which is the intersection of all ideals that contain $H$. Then, $H$ is order dense and majorizing in $G$, and so from Corollary \ref{oc}, $\left.T\right|_{G}$ is order continuous. Thus, we have reduced this case to the previous one.

From part (ii) of Proposition \ref{closure}, if $H$ is a dense regular sublattice of a Banach lattice $F$, it is order dense, and so we have reduced this case to the previous one again.\medskip

Now assume that $F$ is a Banach lattice with the countable supremum property. It is enough to show that if $\left\{f_{n}\right\}_{n\in\N}\subset F_{+}$ are such that $f_{n}\downarrow0_{F}$, then $Tf_{n}\downarrow 0_{E}$. Assume that there is $e>0_{E}$ such that $Tf_{n}\ge e$, for every $n\in\N$. Let $U$ be an open solid neighborhood of $0_{E}$ which does not contain $e$, and let $r>0$ be such that $rT\Bbo_{F}\subset U$. Since $H$ is dense, for every $m,n\in\N$ there is $h_{mn}\in H_{+}$ such that $\|f_{n}-h_{mn}\|<\frac{r}{2^{m+n}}$. As $F$ is a Banach lattice, there is $g=\sum\limits_{m,n\in\N}\left|f_{n}-h_{mn}\right|$ with $\|g\|\le r$, from where $Tg\in U$.

For every $m,n\in\N$ we have $Th_{mn}\ge 0_{E}$ and $Th_{mn}\ge Tf_{n}-\left|Tf_{n}-Th_{mn}\right|\ge e-Tg$, from where $Th_{mn}\ge \left(e-Tg\right)^{+}= e-Tg\wedge e$. Note that since $Tg\in U$, and the latter is solid with $e\notin U$, it follows that $e-Tg\wedge e>0_{E}$. Hence, $\inf\limits_{m,n\in\N}Th_{mn}\ne 0_{E}$ (it either does not exist or is greater than $0_{E}$), and since $\left.T\right|_{H}$ is order continuous, we have $\inf\limits_{m,n\in\N}~_{H}h_{mn}\ne 0_{F}$. Therefore, there is $h\in H$ such that $h_{mn}\ge h>0_{F}$, for every $m,n\in\N$. Since $h_{mn}\xrightarrow[m\to\8]{}f_{n}$, it follows that $f_{n}\ge h$, for every $n\in\N$, but this contradicts the assumption $\inf\limits_{n\in\N}f_{n}=0_{F}$.
\end{proof}

\begin{remark}Note that from part (i) of Proposition \ref{closure}, if $H$ is a dense ideal in $F$, then it is order dense, and so the condition (ii) in the corollary is more general than the condition (i). In condition (iv) the requirement for $F$ to be a Banach lattice can be replaced with the requirement for $E$ to be a Banach lattice (the proof is similar).\qed\end{remark}

Since regularity of a sublattice is equivalent to order continuity of its inclusion, we get the following corollary.

\begin{corollary}\label{nreg}If $E$ is a sublattice of a normed lattice $F$, then $E$ is regular if and only if $\overline{E}$ is regular and $E$ is order dense in $\overline{E}$, provided that $F$ is a Banach lattice, or $E$ is an ideal in $\overline{E}$.
\end{corollary}
\begin{proof}
Sufficiency follows from Remark \ref{tran}. From the same remark, if $E$ is regular in $F$, it is also regular in $\overline{E}$. Hence, $\overline{E}$ is regular, from parts (i) and (iii) of Theorem \ref{choc}, applied to the inclusion of $\overline{E}$ into $F$.
\end{proof}

It would be desirable to drop all additional assumptions from the corollary:

\begin{question}For which topological vector lattices the closure of a regular sublattice is regular?
\end{question}

\begin{remark}
For an example of a non-normed topological vector lattice with this property see part (ii) of Theorem \ref{sublo} below.\qed\end{remark}

\begin{remark}If $E$ is an order complete vector lattice, all results of the section up to this point are valid for order bounded (=regular) operators, since order continuity of $T:F\to E$ is equivalent to order continuity of $\left|T\right|$ (see \cite[Theorem 1.56]{ab}), and the operations $T\to \left|T\right|$ and $T\to \left.T\right|_{H}$ commute for any ideal $H$ of $F$, due to the Riesz-Kantorovich formulas (see \cite[Corollary 1.3.4]{mn}).\qed\end{remark}

\begin{question}Is Proposition \ref{lo} valid for regular operators without the assumption that $E$ is order complete?
\end{question}

Let us now characterize order continuous homomorphisms. Note that the conditions from Corollary \ref{oc} can be added to the list in the following theorem.

\begin{theorem}\label{hoc}For a homomorphism $T:F\to E$ between Archimedean vector lattices the following conditions are equivalent:
\item[(i)] $T$ is order continuous;
\item[(ii)] $T$ preserves supremums of arbitrary sets;
\item[(iii)] $\Ker T$ is a band and $TF$ is a regular sublattice in $E$;
\item[(iv)] $T^{-1}H$ is a band in $F$, for every band $H$ in $E$;
\item[(v)] $T\left(G^{dd}\right)\subset \left(TG\right)^{dd}$, for every $G\subset F$.\medskip
\item[(v')] $T\left(G^{dd}\right)\subset \left(TG\right)^{dd}$, for every ideal $G\subset F$.\medskip

If additionally there is $f\in F_{+}$ such that $\dim TF_{f}=\8$ (e.g. if $F$ has the $\sigma$-property and $\dim TF=\8$), then the conditions above are equivalent to
\item[(vi)] $TH$ is regular in $F$, for every regular sublattice $H$ of $F$;
\item[(vi')] $TH$ is regular in $F$, for every order dense sublattice $H$ of $F$.
\end{theorem}

\textbf{We will only partially prove the theorem here. }Namely, we will prove equivalence of the conditions (i)$\Leftrightarrow$(ii)$\Leftrightarrow$(iii), as well as (i)$\Rightarrow$(iv)$\Rightarrow$(v)$\Rightarrow$(v') and (i)$\Rightarrow$(vi)$\Rightarrow$(vi'), while (v')$\Rightarrow$(i) and (vi')$\Rightarrow$(i) require certain additional machinery from the upcoming sections. Also note that (vi)$\Rightarrow$(i) does not hold without an additional assumption, since there are homomorphisms from vector lattices into $\R$ which are not order continuous.

\begin{proof}
First, note that (ii)$\Rightarrow$(i), (v)$\Rightarrow$(v') and (vi)$\Rightarrow$(vi') are trivial, while (i)+(iii)$\Rightarrow$(vi) follows from the fact that a restriction of an order continuous operator to a regular sublattice is order continuous.

(i)$\Rightarrow$(ii): For $G\subset F$ the collection $G^{\vee}$ of all finite supremums of elements of $G$ is an increasing net with the same supremum (which exists or not for $G$ and $G^{\vee}$ simultaneously). It is easy to see that $TG^{\vee}=\left(TG\right)^{\vee}$, and so if $g=\sup G$, from order continuity we have $Tg=\sup TG^{\vee}=\sup \left(TG\right)^{\vee}=\sup TG$.\medskip

(ii)$\Rightarrow$(iii): We know that $\Ker T$ is an ideal, and if $\Ker T \ni f_{\gamma}\uparrow f$, then $0_{E}=Tf_{\gamma}\uparrow Tf$, from where $Tf=0_{E}$, and so $\Ker T$ is a band. Let us show that $H=TF$ is regular in $F$. Let $G\subset H_{+}$ and $g=\sup_{H} G\in H_{+}$. Let $f\in F_{+}$ be such that $Tf=g$ and let $B=\left(T^{-1} G\right)\cap \left[0_{F},f\right]$. From Remark \ref{homo} $TB=G$, and so once we show that $f=\sup B$, from (ii) we will get $g=Tf=\sup_{E}G$.\medskip

Let $h\in \left[0_{F},f\right]$ be such that $B\le h$. Since $Th\in H$ and $\sup_{H}G=g=Tf\ge Th\ge TB=G$, we get $Th=g=Tf$. Since $0_{F}\le h\le f$ and $B\le h$, it follows that $B+f-h\subset \left[0_{F},f\right]$. At the same time since $Th=Tf$, we have $T\left(b+f-h\right)=Tb\in G$, for any $b\in B$. Hence, $B+f-h\subset B$, and iterating this yields $B+n\left(f-h\right)\subset B\subset \left[0_{F},f\right]$, for every $n\in\N$. As $B$ is nonempty, this implies that $n\left(f-h\right)\le f$, for every $n\in\N$. Since $F$ is Archimedean it follows that $f\le h$, and so $f=h$.\medskip

(iii)$\Rightarrow$(i): Since $TF$ is regular, it is enough to show that if $f_{\gamma}\downarrow 0_{F}$ then $Tf_{\gamma}\downarrow_{TF} 0_{E}$, and so we may assume that $T$ is a surjection. Suppose that $Tf_{\gamma}\ge e\ge 0_{E}$, for every $\gamma$. Since $T$ is a surjective homomorphism, there is $f\in F_{+}$ such that $Tf=e$. For $\gamma\in \Gamma$ let $g_{\gamma}=f-f\wedge f_{\gamma}$; we have $f_{\gamma}\ge f- g_{\gamma}\ge 0_{F}$, from where $0_{F}=\inf \limits_{\gamma\in \Gamma} f_{\gamma}\ge \inf\limits_{\gamma\in \Gamma} \left(f-g_{\gamma}\right)\ge 0_{F}$, and so $0_{F}=\inf\limits_{\gamma\in \Gamma} \left(f-g_{\gamma}\right)=f-\sup\limits_{\gamma\in \Gamma} g_{\gamma}$. On the other hand, $Tg_{\gamma}=e-e\wedge Tf_{\gamma}=0_{E}$, for every $\gamma\in \Gamma$, from where $g_{\gamma}\in \Ker T$, and so $f=\sup\limits_{\gamma\in \Gamma} g_{\gamma}\in \Ker T$. Thus, $e=Tf=0_{E}$, and so $Tf_{\gamma}\downarrow 0_{E}$.\medskip

(i)$\Rightarrow$(iv): It is clear that $T^{-1}H$ is an ideal in $F$. If $f_{\gamma}\uparrow f$, where $f_{\gamma}\in T^{-1}H$, then from order continuity $Tf_{\gamma}\uparrow Tf$, and $Tf_{\gamma}\in H$. Since $H$ is a band, it contains all supremums, and so $Tf\in H$, from where $f\in T^{-1}H$. Thus, $T^{-1}H$ contains all supremums, and so it is a band.\medskip

(iv)$\Rightarrow$(v): Since $\left(TG\right)^{dd}$ is a band in $E$ that contains $TG$, it follows from our assumption that $T^{-1}\left(TG\right)^{dd}$ is a band in $F$ that contains $T^{-1}TG\supset G$. Hence $T^{-1}\left(TG\right)^{dd}$ contains $G^{dd}$, and so $T\left(G^{dd}\right)\subset TT^{-1}\left(TG\right)^{dd}\subset \left(TG\right)^{dd}$.
\end{proof}

\begin{remark}\label{arch}
Note that the only implication in Corollary \ref{oc} that requires the Archimedean property is (iv)$\Rightarrow$(i), and the only established implication in Theorem \ref{hoc} that requires the Archimedean property is (ii)$\Rightarrow$(iii). Moreover, in the absence of this property the latter implication still holds if we additionally assume that $\Ker T$ is a projection band. Indeed, then for $H=\left(\Ker T\right)^{d}$ the restriction $\left.T\right|_{H}$ is an isomorphism of $H$ onto $TF$, from where the inclusion of $TF$ into $E$ is a composition $T\left.T\right|_{H}^{-1}$ of order continuous operators.
\qed\end{remark}

\begin{question} Does the implication (vi)$\Rightarrow$(i) in Theorem \ref{hoc} hold under the mere assumption that $\dim TF=\8$?
\end{question}

\begin{question} Is there a version of the equivalency (i)$\Leftrightarrow$(v) in Theorem \ref{hoc} for positive operators?
\end{question}

The conditions in Theorem \ref{hoc} can be translated into the conditions of regularity of a sublattice of an Archimedean vector lattice.

\begin{corollary}\label{reg}For a sublattice $E$ of $F$ the following conditions are equivalent:
\item[(i)] $E$ is regular;
\item[(ii)] $E\cap H$ is regular in $H$, for every ideal $H$ of $F$;
\item[(iii)] There is a majorizing $G\subset E_{+}$ such that $E\cap F_{g}$ is regular in $F_{g}$, for every $g\in G$;
\item[(iv)] $E\cap H$ is a band in $E$, for every band $H$ in $F$;
\item[(v)] $G^{dd}_{F}\cap E= G^{dd}_{E}$, for every $G\subset E$;
\item[(v')] $G^{dd}_{F}\cap E= G^{dd}_{E}$, for every ideal $G\subset E$;
\item[(vi)] There is an order dense and majorizing sublattice $H$ of $E$ which is regular in $F$.
\end{corollary}
\begin{proof}
Equivalence of (i), (iii) and (vi) follows from Corollary \ref{oc}, while equivalence of (i), (iv), (v) and (v') follows from Theorem \ref{hoc}. (ii)$\Rightarrow$(i) is trivial, while the converse follows from the fact that $E\cap H$ is an ideal in $E$, therefore it is regular in $E$, from where it is regular in $F$, and so in $H$ (see remarks \ref{tran} and \ref{menag}).
\end{proof}

Again, we caution that only (vi)$\Leftrightarrow$(i)$\Leftrightarrow$(ii)$\Leftrightarrow$(iii)$\Rightarrow$(iv)$\Rightarrow$(v)$\Rightarrow$(v') has been proven at this point, but apart from (vi)$\Rightarrow$(i) none of these implications require the Archimedean property, due to Remark \ref{arch}.

\section{Preliminaries on $\Co\left(X\right)$}\label{x}

In this section we gather various well known results and concepts related to the lattices of continuous functions. Everywhere in this section $X$ is (at least) a Tychonoff topological space.

We denote the space of all continuous functions on $X$ by $\Co\left(X\right)$. This space is a vector lattice with respect to the pointwise operations, and so it is a sublattice of the vector lattice $\Fo\left(X\right)$ of all real-valued functions on $X$. It is easy to see that these lattices are Archimedean. The interval $\left[f,g\right]$, where $f,g\in \Fo\left(X\right)$ will be meant in $\Co\left(X\right)$ as a sublattice of $\Fo\left(X\right)$, i.e. all continuous functions between (not necessarily continuous) $f$ and $g$. We will denote the function which is identically $0$ on $X$ as $\0$, while $\1_{A}$ is the indicator of $A\subset X$, i.e. the function whose value is $1$ on $A$ and $0$ on $X\backslash A$. If $X$ is clear from the context, we put $\1=\1_{X}$.

The space $\Co\left(X\right)$ is by default endowed with the compact-open topology. One can show that this topology is locally convex-solid. Let $\Co_{b}\left(X\right)$ stand for the Banach lattice of all bounded continuous functions on $X$ endowed with the supremum norm $\|\cdot\|$. The closed unit ball of $\Co_{b}\left(X\right)$ is $\left[-\1,\1\right]$. It is easy to see that $f\in \Co_{b}\left(X\right)$ is a strong unit if and only if $f\ge\delta\1$, for some $\delta>0$.

If $Y$ is another Tychonoff space and $\varphi:X\to Y$ is a continuous map, the \emph{composition operator} $C_{\varphi}:\Co\left(Y\right)\to \Co\left(X\right)$ is defined by $C_{\varphi}f=f\circ\varphi$. It is easy to see that $C_{\varphi}$ is a continuous homomorphism. Even though formally the analogously defined composition operator from $\Co_{b}\left(Y\right)$ into $\Co_{b}\left(X\right)$ is distinct from $C_{\varphi}$, we will still denote it by $C_{\varphi}$ unless there is a risk of confusion. A special case of a composition operator is the restriction operator, i.e. the case when $X\subset Y$ and $\varphi$ is the inclusion map. It is a standard fact that if $Y=\beta X$ is the Stone-Cech compactification of $X$, then $C_{\varphi}$ is an isomorphism from $\Co\left(\beta X\right)$ onto $\Co_{b}\left(X\right)$. Another important class of operators is formed by \emph{multiplication operators} of the form $M_{f}:\Co\left(X\right)\to \Co\left(X\right)$, where $f\in \Co\left(X\right)$, and is defined by $M_{f}g=fg$. It is easy to see that if $f\ge\0$, then $M_{f}$ is a continuous homomorphism.\medskip

By definition of the Tychonoff space, for every open $U\subset X$ and $x\in U$, there is $f\in \left[\1_{\left\{x\right\}},\1_{U}\right]$. Equivalently, $\sup_{\Fo\left(X\right)} \Co\left(X\right)\cap \left[\0,\1_{U}\right]=\1_{U}$. Moreover, if $K$ is compact, or if $X$ is normal and $K\subset U$ is closed, Tietze-Urysohn theorem guarantees that there is $f\in \left[\1_{K},\1_{U}\right]$ (see \cite[theorems 2.1.8 and 3.1.7]{engelking}). It turns out we can achieve this ``richness'' even for some sublattices of $\Co\left(X\right)$.

\begin{proposition}[Sublattice Urysohn lemma]
Let $E$ be a dense sublattice of $\Co_{b}\left(X\right)$ that contains $\1$, let $U\subset X$ be open and let $K\subset U$ be compact. Then $E\cap \left[\1_{K},\1_{U}\right]\ne\varnothing$. If additionally $X$ is normal, then $K$ can be assumed to be merely closed. In this case $E$ contains all simple continuous functions.
\end{proposition}
\begin{proof}
From Tietze-Urysohn theorem there is $g\in \Co\left(X\right)$ such that $\1_{K}\le g \le \1_{U}$. Since $E$ is dense, there is $h\in E$ such that $\|h-g\|\le\frac{1}{3}$. The latter is equivalent to $-\frac{1}{3}\1\le h-g\le \frac{1}{3}\1$, or $3g-2 \1\le 3 h-\1 \le 3g$. Let $f= \left(3 h-\1\right)^{+}\wedge \1\in E$. Then $f\le 3 g^{+}\wedge \1\le 3 \1_{U}\wedge \1=\1_{U}$, and simultaneously $f\ge \left(3g-2 \1\right)^{+}\wedge \1\ge \left(3\1_{K}-2 \1\right)^{+}\wedge \1=\1_{K}\wedge \1=\1_{K}$.

If $X$ is normal and $K$ is a clopen set, then $\left[\1_{K},\1_{K}\right]=\left\{\1_{K}\right\}$ intersects with $E$, and so $\1_{K}\in E$. Since any simple continuous function is a linear combination of indicators of clopen sets, it follows that $E$ contains all simple continuous functions.
\end{proof}

Every dense sublattice $E$ of $\Co_{b}\left(X\right)$ is unital, since there is $e\in E$ such that $\|\1-e\|<\frac{1}{2}$, from where $e\ge \frac{1}{2}\1$. Also, the norm of $M_{f}$ on $\Co_{b}\left(X\right)$ is equal to $\|f\|$, and so $M_{f}$ is a continuous automorphism of $\Co_{b}\left(X\right)$ if and only if $f$ is a strong unit. This leads to the following generalization of the result above.

\begin{corollary}\label{sul}If $E$ is a dense sublattice of $\Co_{b}\left(X\right)$, then $E\cap \left[\1_{K},\delta\1_{U}\right]\ne\varnothing$, for every open $U\subset X$, compact $K\subset U$ (if $X$ is normal, it suffices to assume that $K$ is closed) and $\delta>1$.
\end{corollary}
\begin{proof}
Fix $\varepsilon\in\left(0,1\right)$ and let $f\in E$ be such that $\|\1-f\|<\varepsilon$, and so $\left(1-\varepsilon\right)\1\le f\le \left(1+\varepsilon\right)\1$. Then, $M_{f}$ is an automorphism of $\Co_{b}\left(X\right)$, from where $F=M_{f}^{-1}E$ is a dense sublattice of $\Co_{b}\left(X\right)$ that contains $\1$. From Sublattice Urysohn lemma there is $g\in F\cap \left[\1_{K},\1_{U}\right]$. Therefore, $fg\in E$ satisfies $\left(1-\varepsilon\right)\1_{K}\le fg\le \left(1+\varepsilon\right)\1_{U}$, from where $\1_{K}\le \frac{1}{1-\varepsilon}fg\le \frac{1+\varepsilon}{1-\varepsilon}\1_{U}$. Taking $\varepsilon$ such that $\frac{1+\varepsilon}{1-\varepsilon}\le\delta$ completes the proof.
\end{proof}

We will see below that the statement of Corollary \ref{sul} in general does not hold for dense sublattices of $\Co\left(X\right)$. On the other hand, since the set of restrictions of elements of such sublattices to a compact set $K\subset X$ forms a dense sublattice of $\Co\left(K\right)$, we get the following result.

\begin{corollary}\label{ssul}
Let $E$ be a dense sublattice of $\Co\left(X\right)$ and let $K,L\subset X$ be compact and disjoint. Then, for every $\delta>1$ there is $f\in E_{+}$ which vanishes on $L$ and such that $f\left(K\right)\subset \left[1,\delta\right]$.
\end{corollary}

A \emph{point evaluation} $\delta_{x}$ at $x\in X$ is the linear functional on $\Co\left(X\right)$ defined by $\delta_{x}\left(f\right)=f\left(x\right)$. If $E$ is a subspace of $\Co\left(X\right)$ we denote the restriction of $\delta_{x}$ to $E$ by $\delta_{x}^{E}$. We will say that a subspace $E$ of $\Co\left(X\right)$:
\begin{itemize}
\item \emph{vanishes} at $x\in X$ if $f\left(x\right)=0$, for every $f\in E$, i.e. $\delta_{x}^{E}=0_{E^{*}}$;
\item \emph{separates} $x,y\in X$ if there is $f\in E$ such that $f\left(x\right)\ne f\left(y\right)$, i.e. $\delta_{x}^{E}\ne \delta_{y}^{E}$;
\item \emph{strictly separates} $x,y\in X$ if there are $f,g\in E$ such that $f\left(x\right)=0$, $f\left(y\right)=1$, $g\left(x\right)=1$ and $g\left(y\right)=0$, i.e. $\delta_{x}^{E}$ and $\delta_{y}^{E}$ are not linearly dependent;
\item \emph{(strictly) separates points [of $Y\subset X$]} if it (strictly) separates every pair of distinct points [from $Y$];
\item \emph{has only simple constraints} if whenever $\delta_{x}^{E}$ and $\delta_{y}^{E}$ are linearly dependent, either one of them is $0_{E^{*}}$, or $\delta_{x}^{E}= \delta_{y}^{E}$.
\end{itemize}

We will need the following version of the Stone-Weierstrass theorem.

\begin{theorem}Let $E$ be a sublattice of $\Co\left(X\right)$. Then $f\in \Co\left(X\right)$ belongs to $\overline{E}$ if and only if for every $x,y\in X$ there is $g\in E$ such that $f\left(x\right)=g\left(x\right)$ and $f\left(y\right)=g\left(y\right)$. In particular, a sublattice $F$ of $\Co\left(X\right)$ is dense if and only if it strictly separates points.
\end{theorem}
\begin{proof}
Necessity: Point evaluations are continuous linear functionals, and so they are (not) linearly independent on $E$ and $\overline{E}$ simultaneously. Hence, any system of two linear equations has (not) solution in $E$ and $\overline{E}$ simultaneously.

Sufficiency: Let $K\subset X$ be compact and let $\varepsilon>0$. Let $F$ be the set of restrictions of elements of $E$ to $K$. From the appropriate version of Stone-Weierstrass theorem for (not necessarily linear) sublattices of $\Co\left(K\right)$ (see \cite[Theorem 16.5.5]{bn}), there is $g\in F$ such that $\|\left.f\right|_{K}-g\|<\varepsilon$. Let $h\in E$ be such that $g=\left.h\right|_{K}$. We have $\|\left.f\right|_{K}-\left.h\right|_{K}\|<\varepsilon$, and since $K$ and $\varepsilon$ were chosen arbitrarily, the result follows.
\end{proof}

Note that in particular $\Co_{b}\left(X\right)$ is dense in $\Co\left(X\right)$.\medskip

The Stone-Weierstrass theorem essentially states that $\overline{E}$ is the maximal set which satisfies the same constraints as $E$. This allows to obtain the following description of the closed sublattices of $\Co\left(X\right)$.

\begin{corollary}\label{subl}For every closed sublattice $E$ of $\Co\left(X\right)$ there is a closed $Y\subset X$ and a collection of triples $\left\{\left(x_{i},z_{i},\alpha_{i}\right)\right\}_{i\in I}\subset \left(X\backslash Y\right)^{2}\times \left(0,+\8\right)$ such that $$E=\left\{f\in \Co\left(X\right),~ \forall y\in Y:~ f\left(y\right)=0,~ \forall i\in I:~ f\left(x_{i}\right)=\alpha_{i}f\left(z_{i}\right)\right\}.$$
\end{corollary}

Let $E$ be a sublattice of $\Co\left(X\right)$. For $A\subset X$ define $E_{A}=\left\{f\in E,~\forall x\in A:~ f\left(x\right)=0\right\}$. Clearly, $E_{A}$ is an ideal in $E$. Since elements of $E$ are continuous, it follows that $E_{A}=E_{\overline{A}}$. If $B\subset X$, then $E_{A\cup B}=E_{A}\cap E_{B}$, and if $A\subset B$, then $E_{B}\subset E_{A}$. Observe that $f,g\in \Co\left(X\right)$ are disjoint if and only if $f^{-1}\left(\R\backslash \left\{0\right\}\right)\cap g^{-1}\left(\R\backslash \left\{0\right\}\right)=\varnothing$. Since disjointness is inherited by sublattices, it follows that $\left\{f\right\}^{d}=E_{\overline{X\backslash f^{-1}\left(0\right)}}$, for $f\in E$, and if $G\subset E$, then $G^{d}= E_{\overline{X\backslash \bigcap\limits_{f\in G} f^{-1}\left(0\right)}}$.\medskip

It is easy to see that every ideal of $\Co\left(X\right)$ strongly separates any pair of distinct points on which it does not vanish. This observation together with Corollary \ref{subl} leads the following description of the closed ideals in $\Co\left(X\right)$.

\begin{corollary}\label{id}For every closed ideal $E$ of $\Co\left(X\right)$ there is a closed $A\subset X$ such that $E=\Co\left(X\right)_{A}$.
\end{corollary}

This result is not true for $\Co_{b}\left(X\right)$, since unless $X$ is compact, it has some ``hidden points''.

\begin{example}
Let $X$ be locally compact but not compact. Consider the set $\Co_{0}\left(X\right)$  that consists of the functions that \emph{vanish at infinity}, i.e. all $f\in \Co\left(X\right)$ such that for every $\varepsilon>0$ the set $f^{-1}\left(\left[\varepsilon,+\8\right)\right)$ is compact. It is easy to see that $\Co_{0}\left(X\right)$ is a closed ideal of $\Co_{b}\left(X\right)$, which vanishes nowhere. It turns out that it vanishes at a ``hidden point''.

Let $X_{\8}=X\cup\left\{\8\right\}$ be the one point compactification of $X$, and let $\varphi:X\to X_{\8}$ be the inclusion. It is easy to see that $C_{\varphi}$ is an isometric isomorphism from $\Co\left(X_{\8}\right)_{\left\{\8\right\}}$ onto $\Co_{0}\left(X\right)$.

Applying the Stone-Weierstrass theorem to $\Co\left(X_{\8}\right)_{\left\{\8\right\}}$ we see that a sublattice of $\Co_{0}\left(X\right)$ is dense if and only if it strictly separates points of $X$. Note that if $E$ is a subalgebra, or has a so called truncation (see \cite[Lemma 4.1]{bh}), it is dense if and only if it vanishes nowhere and separates points.
\qed\end{example}

\begin{example}\label{sch}
More generally, it follows from identification of $\Co_{b}\left(X\right)$ with $\Co\left(\beta X\right)$, that a closed ideals of $\Co_{b}\left(X\right)$ are of the form $\Co\left(\beta X\right)_{A}$, where $A\subset \beta X$ is closed. Note however, that if $A\subset X\subset \beta X$, then $\Co_{b}\left(X\right)_{A}$ can be identified with $\Co\left(\beta X\right)_{A}$, since for a function on $\beta X$ to have a restriction to $X$ that vanishes on $A$ simply means to vanish on $A$.
\qed\end{example}

Let us derive another corollary of the Stone-Weierstrass theorem.

\begin{proposition}\label{sd}
Let $E$ be a dense sublattice of $\Co\left(X\right)$ and let $K\subset X$ be compact. Then, $E_{K}$ is dense in $\Co\left(X\right)_{K}$.
\end{proposition}
\begin{proof}
It is enough to show that $E_{K}$ strictly separates points of $X\backslash K$. Indeed, if $x,y\in X\backslash K$, from Corollary \ref{ssul} there is $f\in E$ which vanishes on $K\cup\left\{y\right\}$ (and so $f\in E_{K}$ and $f\left(y\right)=0$) and such that $f\left(x\right)\in\left(1,2\right)$.
\end{proof}\smallskip

It is known that every closed subalgebra of $\Co\left(X\right)$ is a sublattice (the proof of \cite[Theorem 16.5.2]{bn}). Conversely, a closed sublattice of $\Co\left(X\right)$ is a subalgebra if and only if it has only simple constraints, i.e. in Corollary \ref{subl} $\alpha_{i}=1$, for every $i\in I$. Indeed, if $E$ is a subalgebra, and $f\left(x\right)=\alpha f\left(z\right)$, for every $f\in E$, this is also true for $f^{2}$, from where $\alpha\in\left\{0,1\right\}$. Note that such sublattice vanishes nowhere if and only if it contains $\1$.\medskip

Let $E$ be a sublattice of $\Co\left(X\right)$ that has only simple constraints, and which vanishes on (a closed) $A\subset X$. Let $X_{E}$ be $\left\{\delta_{x}^{E},~ x\in X\right\}$ endowed with the weak* topology. From our assumption $X_{E}$ may contain $0_{E^{*}}$, but it cannot contain non-zero linearly dependent elements. It is easy to see that the map $\varphi: X\to X_{E}$ defined by $\varphi\left(x\right)=\delta^{E}_{x}$ is a continuous surjection and $A=\varphi^{-1}\left(0_{E^{*}}\right)$. Moreover, the map $J:E\to\Co\left(X_{E}\right)$ defined by $\left[Jf\right]\left(\delta^{E}_{x}\right)=f\left(x\right)$ is a continuous homomorphism, such that $C_{\varphi}J=Id_{E}$. Hence, $JE$ is a sublattice of $\Co\left(X_{E}\right)$, and $J$ is an isomorphism from $E$ onto $JE$. Moreover, from our assumption, $JE$ strictly separates points of $X_{E}\backslash\left\{0_{E^{*}}\right\}$. Therefore, from the Stone-Weierstrass theorem $\overline{JE}=\Co\left(X_{E}\right)_{\left\{0_{E^{*}}\right\}}$ (if $A=\varnothing$, then $\overline{JE}=\Co\left(X_{E}\right)$). Let us summarize for the case when $X$ is compact.

\begin{proposition}\label{quot}
Let $X$ be compact and let $E$ be a sublattice of $\Co\left(X\right)$ that has only simple constraints. Then, there is a compact space $X_{E}$ and a continuous surjection $\varphi:X\to X_{E}$ such that $C_{\varphi}$ is an isometric isomorphism from $\Co\left(X_{E}\right)$ onto $\overline{E}$, if $E$ vanishes nowhere, and from $\Co\left(X_{E}\right)_{o}$ onto $\overline{E}$, for some $o\in X_{E}$, otherwise.
\end{proposition}

\begin{remark}\label{eqr}Recall that an equivalence relation $\sim$ on $X$ is called \emph{closed} if for every closed $A\subset X$ the union of all classes of $\sim$ that intersect $A$ is closed. This is equivalent to the fact that the quotient map from $X$ onto $X\slash\sim$ is closed (see \cite[Proposition 2.4.9]{engelking}). If $X$ is compact, there is a bijective correspondence between closed equivalence relations on $X$ and surjective maps from $X$ onto compact spaces (see \cite[Proposition 3.2.11]{engelking}). Hence, there is a bijective correspondence between closed equivalence relations on $X$ and closed sublattices that contain $\1$.

Also note that in the class of closed sublattices of $\Co\left(X\right)$ that have only simple constraints vanishing nowhere is not a restrictive property. Indeed, if $E$ is a closed sublattice that vanishes on $A$ and has only simple constraints, it is a subalgebra, and so $F=E+\R \cdot\1$ is a closed subalgebra, and therefore a sublattice with $E=F_{A}$. Hence, there is a bijective correspondence between closed equivalence relations on $X$ with a distinguished class and closed sublattices of $\Co\left(X\right)$ that have only simple constraints. More discussion on this topic see in the end of Section \ref{c}.
\qed\end{remark}\smallskip

Let us conclude the section with discussing the notions of supremum and infimum in $\Co\left(X\right)$. For every $g\in \Co\left(X\right)_{+}\backslash \left\{\0\right\}$ there is an open nonempty $U\subset X$ and $\varepsilon>0$ such that $g\ge \varepsilon\1_{U}$. Conversely, for every open nonempty $U\subset X$ there is $g\in \left(\0,\1_{U}\right]$ (indeed, pick any $g\in \left[\1_{\left\{x\right\}},\1_{U}\right]$, for some $x\in U$). The following is a version of a result from \cite{kv1}.

\begin{proposition}\label{inf}
If $G\subset \Co\left(X\right)_{+}$, then $\inf G=\0$ if and only if for every $n\in\N$ and every open $U$ in $X$ there are $f\in G$ and $x\in U$ such that $f\left(x\right)<\frac{1}{n}$.
\end{proposition}
\begin{proof}
Sufficiency: Let $g\in \Co\left(X\right)$ be such that $\0\le g\le G$. For every open nonempty $U\subset X$ and $n\in\N$ there is $f\in G$ and $x\in X$ such that $g\left(x\right)\le f\left(x\right)<\frac{1}{n}$. Hence, from the comment before the proposition, $g=\0$.

Necessity: Assume that there is $n\in\N$ and open nonempty $U$ such that $f\left(x\right)\ge\frac{1}{n}$, for every $f\in G$ and $x\in U$. Take $g\in \left(\0,\1_{U}\right]$; we have $\0<\frac{1}{n}g\le f$, for every $f\in G$. Contradiction.
\end{proof}

Even though the following corollary will not be used in the sequel, we find it appropriate to present it here. The first part is a version of a result from \cite{kv1}, the second part for sequences was considered much more comprehensively in \cite{vdw}.

\begin{corollary}
\item[(i)] Let $G\subset \Co\left(X\right)_{+}$. If there is a dense $Y\subset X$, such that $\inf\limits_{f\in G} f\left(x\right)=0$, for every $x\in Y$, then $\inf G=\0$. The converse holds if $X$ is a Baire space (including locally compact or metrizable by a complete metric).
\item[(ii)] If $X$ is a Baire space, then the uo-convergence in $\Co\left(X\right)$ implies pointwise convergence on a dense subset of $X$.
\end{corollary}
\begin{proof}
(i): The first claim follows immediately from Proposition \ref{inf}. For the second one consider $U_{n}=\bigcup\limits_{f\in G}f^{-1}\left(\frac{1}{n},+\8\right)$, where $n\in\N$. It is an open set, which is dense, according to Proposition \ref{inf}. Since $X$ is a Baire space, the set $\left\{x\in X,~\inf\limits_{f\in G} f\left(x\right)=0\right\}=\bigcap\limits_{n\in\N} U_{n}$ is dense.

(ii): Let $\left\{f_{\gamma}\right\}_{\gamma\in \Gamma}\subset \Co\left(X\right)$ be an uo-null net, and let $\left\{g_{\beta}\right\}_{\beta\in B}\subset \Co\left(X\right)$ be decreasing to $\0$ and such that for every $\beta\in B$ there is $\gamma\in \Gamma$ such that $\left|f_{\alpha}\right|\wedge \1\le \left|g_{\beta}\right|$, as soon as $\alpha\ge \gamma$. Since $g_{\beta}\downarrow \0$, from (i) there is a dense set $Y\subset X$ such that $g_{\beta}\left(y\right)\downarrow 0$, for every $y\in Y$. But then $f_{\gamma}\left(y\right)\to 0$, for every $y\in Y$.
\end{proof}

\section{Sublattices of $\Co\left(X\right)$}\label{s}

Let us investigate ``richness'' of a sublattice of $\Co\left(X\right)$. We again assume that $X$ is Tychonoff throughout the section.  We will call a sublattice $E\subset\Co\left(X\right)$ an \emph{Urysohn sublattice} if for every open nonempty $U\subset X$ and $x\in U$, there is $f\in E$ such that $\left.f\right|_{X\backslash U}\equiv 0$, and $f\left(x\right)\ne 0$, i.e. $E_{X\backslash U}$ vanishes at no point of $U$. We will also call $E$ a \emph{weakly Urysohn sublattice} if for every open nonempty $U\subset X$, there is $f\in E$ such that $\left.f\right|_{X\backslash U}\equiv 0$, and $f>\0$, i.e. $E_{X\backslash U}\ne\left\{\0\right\}$. The following is straightforward.

\begin{itemize}
\item $E$ is an Urysohn sublattice if and only if  $\overline{A}=\bigcap\limits_{f\in E_{A}}f^{-1}\left(0\right)$, for any $A\subset X$, and if and only if $E_{B}\subset E_{A}\Rightarrow \overline{A}\subset \overline{B}$, for any $A,B\subset X$.
\item $E$ is a weakly Urysohn sublattice if and only if $\Int \overline{A}=\Int \bigcap\limits_{f\in E_{A}}f^{-1}\left(0\right)$, for every $A\subset X$, and if and only if $E_{B}\subset E_{A}\Rightarrow \Int\overline{A}\subset \Int\overline{B}$, for any $A,B\subset X$.
\end{itemize}

It follows from the Stone-Weierstrass theorem that every Urysohn sublattice is dense in $\Co\left(X\right)$. Since $X$ is a Tychonoff space, both $\Co\left(X\right)$ and $\Co_{b}\left(X\right)$ are Urysohn sublattices. If a sublattice contains a (weakly) Urysohn sublattice, it is also a (weakly) Urysohn sublattice.\medskip

If $E$ is weakly Urysohn, then $E_{A}^{d}=E_{\overline{X\backslash \bigcap\limits_{f\in E_{A}} f^{-1}\left(0\right)}}=E_{X\backslash \Int\overline{A}}=E_{\overline{X\backslash \overline{A}}}$, from where $E_{A}^{dd}=E_{\overline{X\backslash X\backslash \Int\overline{A}}}=E_{\overline{\Int\overline{A}}}$, and $\left\{f\right\}^{dd}=E_{\overline{\Int f^{-1}\left(0\right)}}$, for every $f\in E$. Since in a Urysohn sublattice we can recover $A$ from $E_{A}$, these equalities allow to obtain the following characterizations.

\begin{proposition}\label{ur}Let $E\subset\Co\left(X\right)$ be a Urysohn sublattice. Then:
\item[(i)]  $E_{A}$ is a band in $E$ if and only if $\overline{A}=\overline{\Int\overline{A}}$, i.e. $\overline{A}$ is a closure of an open set (also known as a regularly closed set).
\item[(ii)] If $E_{A}$ is a projection band in $E$, then $\overline{A}$ is clopen.
\item[(iii)] If $E$ has PP, then $X$ is extremally disconnected.
\item[(iv)] If $E$ has PPP, then $X$ is totally disconnected.
\end{proposition}
\begin{proof}
(i) follows from $E_{A}^{dd}=E_{\overline{\Int\overline{A}}}$.

(ii): If $\overline{A}$ is not clopen there is $x\in \partial \overline{A}=\left(X\backslash \Int\overline{A}\right)\cap \overline{A}$, therefore $E_{\overline{A}}, E_{X\backslash \Int\overline{A}}\subset E_{\left\{x\right\}}\ne E$, and so $E\ne E_{\overline{A}}+E_{\overline{A}}^{d}$. Contradiction.

(iii) follows from combining (i) and (ii).

(iv): Let $x,y\in X$ be distinct. Let $U$ be a neighborhood of $x$ such that $y\not\in\overline{U}$. There is $f\in E$ that vanishes outside $U$, and such that $f\left(x\right)\ne 0$. Then, $\left\{f\right\}^{dd}=E_{\overline{\Int f^{-1}\left(0\right)}}$, and the latter is a projection band, from where $\overline{\Int f^{-1}\left(0\right)}$ is clopen. Since $f\left(x\right)=1$, we have $x\not\in f^{-1}\left(0\right)\supset \overline{\Int f^{-1}\left(0\right)}$. On the other hand, since $f$ vanishes outside of $U$, we have $y\in X\backslash\overline{U}\subset \Int f^{-1}\left(0\right)$. Hence, we found a clopen set that contains exactly one of an arbitrary pair of points in $X$. Thus, $X$ is totally disconnected.\end{proof}

Recall that from Corollary \ref{id}, every closed ideal in $\Co\left(X\right)$ is of the form $\Co\left(X\right)_{A}$, for some closed $A\subset X$. Since $\Co\left(X\right)$ is an Urysohn sublattice, it follows that $A$ is unique. We also get the following description of bands and projection bands in $\Co\left(X\right)$.

\begin{corollary}\label{bpb}
$H$ is a (projection) band in $\Co\left(X\right)$ if and only if $H=\Co\left(X\right)_{A}$, for some regularly  closed (clopen) $A\subset X$.
\end{corollary}
\begin{proof}
It follows from Corollary \ref{id}, and parts (i) and (ii) of Proposition \ref{ur} that we only need to prove that $\Co\left(X\right)_{A}$ is a projection band, for every clopen $A\subset X$. Indeed, both $\1_{A}$ and $\1_{X\backslash A}$  are continuous, and so $f=\1_{A}f + \1_{X\backslash A}f$, for every $f\in \Co\left(X\right)$, from where $\Co\left(X\right)=\Co\left(X\right)_{A}+\Co\left(X\right)_{A}^{d}$.
\end{proof}

We now move on to study regularity and order density of sublattices of $\Co\left(X\right)$.

\begin{proposition}\label{sub}Let $E\subset \Co\left(X\right)$ be a sublattice. Then:
\item[(i)] $E$ is regular if and only if $\inf_{E}\left\{f\in E,~f\ge \1_{U}\right\}\ne\0$ (this includes this set being empty, or having no infimum), for every open nonempty $U\subset X$.
\item[(ii)]  $E$ is order dense if and only if for every open nonempty $U\subset X$ there is $f\in E$ such that $\0<f\le \1_{U}$, and if and only if $\bigcap \limits_{f\in E\cap \left[\0,\1_{U}\right]}f^{-1}\left(0\right)$ is nowhere dense in $U$, for every open nonempty $U\subset X$.
\item[(iii)] If $A\subset X$, then $E_{A}$ is order dense if and only if $E$ is order dense and $A$ is nowhere dense.
\end{proposition}
\begin{proof}
(i): Assume that $E$ is regular and $U\subset X$ is open nonempty and such that $G=\left\{f\in E,~f\ge \1_{U}\right\}\ne\varnothing$. If $\inf_{E}G=\0$, then, $\inf_{\Co\left(X\right)}G=\0$. On the other hand, there is $g\in \Co\left(X\right)$ such that $\0< g\le \1_{U}\le G$. Contradiction.

Conversely, assume $f_{\gamma}\downarrow_{E} \0$, but there is $g\in \Co\left(X\right)$ such that $\0<g\le f_{\gamma}$, for every $\gamma$. There is an open nonempty $U\subset X$ and $\varepsilon>0$ such that $g\ge \varepsilon\1_{U}$. Then $\frac{1}{\varepsilon}f_{\gamma}\ge \frac{1}{\varepsilon}g \ge \1_{U}$, and so $\frac{1}{\varepsilon}f_{\gamma}\in G$. Hence, $\0\le \inf_{E}G\le  \frac{1}{\varepsilon}\inf\limits_{\gamma\in \Gamma}f_{\gamma}=\0$.

(ii): The first equivalence follows from the observations before Proposition \ref{inf}. The last condition is a reformulation of the second one.

(iii): Necessity: Since $E_{A}\subset E$, order density of the former yields order density of the latter. If $U=\Int \overline{A}\ne \varnothing$, then there is no $f\in E_{A}$ such that $\0<f\le \1_{U}$. Sufficiency: if $U\subset X$ is open and nonempty, then $V=U\backslash \overline{A}$ is also open and nonempty, and since $E$ is order dense, there is $f\in E$ such that $\0<f\le \1_{V}$. Then $f\in E_{A}$ and $\0<f\le \1_{U}$, and so from (ii) $E_{A}$ is order dense.
\end{proof}

It follows from part (ii) that every order dense sublattice is weakly Urysohn and that $\Co_{b}\left(X\right)$ is order dense in $\Co\left(X\right)$. Note that in (i) and (ii) it is enough to check that the property holds for every $U$ from a certain base of the topology of $X$. If $X$ is locally compact, one can choose this base to consist of relatively compact regularly open ($U$ is regularly open if $U=\Int \overline{U}$) sets. Therefore, in this case the preceding result allows a refinement.

\begin{theorem}\label{sublo}Let $X$ be locally compact and let $E\subset \Co\left(X\right)$ be a sublattice. Then:
\item[(i)] $E$ is a weakly Urysohn sublattice if and only if it is order dense.
\item[(ii)] $E$ is regular in $\Co\left(X\right)$ if and only if $\overline{E}$ is regular and $E$ is order dense in $\overline{E}$.
\end{theorem}
\begin{proof}
(i): We only need to prove necessity. Let $U\subset X$ be open nonempty and relatively compact. Since $E$ is weakly Urysohn, there is $f\in E_{+}\backslash\left\{\0\right\}$ which vanishes outside of $U$. Since $U$ is relatively compact, there is $x\in U$ such that $0<a=f\left(x\right)=\|f\|$. But then $\0<\frac{1}{a}f\le \1_{U}$. Since $U$ was chosen arbitrarily, from part (ii) of Proposition \ref{sub} we conclude that $E$ is order dense.\medskip

(ii): In the light of Remark \ref{tran}, it is enough to prove necessity. Let us start with regularity of $\overline{E}$.

For an open $U\subset X$ let $G_{U}^{E}=\left\{f\in E,~f\ge \1_{U}\right\}$. From the comment before the proposition and part (i) of Proposition \ref{sub}, it is given that $\inf_{E}G^{E}_{U}\ne\0$, for every open nonempty relatively compact $U\subset X$, and we need to show that $\inf_{\overline{E}}G^{\overline{E}}_{U}\ne\0$, for every open nonempty relatively compact $U\subset X$.

Let $H=\left\{f\in \Co\left(X\right),~\forall x\in\overline{U},~  f\left(x\right)>1\right\}$, which is an open set in $\Co\left(X\right)$, since $\overline{U}$ is compact. Clearly, $H\cap E_{+}\subset G^{E}_{U}$, but we also have $G_{U}\subset\overline{H\cap E_{+}}$, since every $f\in G^{E}_{U}$ is the limit of the sequence $\left\{\frac{n+1}{n}f\right\}_{n\in\N}\subset H\cap E_{+}$. Moreover, since $H$ is open, we have $\overline{G^{E}_{U}}=\overline{H\cap E_{+}}=\overline{H\cap \overline{E}_{+}}=\overline{G^{\overline{E}}_{U}}$.

Assume that $U$ is such that $G^{\overline{E}}_{U}\ne\varnothing$. Then, $G^{E}_{U}\ne\varnothing$ and since $E$ is regular, from part (ii) of Proposition \ref{sub}, there is $g\in E\subset \overline{E}$ such that $\0<g\le G^{E}_{U}$. As the set $\left\{f\in \Co\left(X\right),~f\ge g\right\}$ is closed in $\Co\left(X\right)$, we have that $g\le \overline{G^{E}_{U}}\supset G^{\overline{E}}_{U}$, and so $\inf_{\overline{E}}G^{\overline{E}}_{U}\ne\0$.\medskip

Let us now show that $E$ is order dense in $\overline{E}$. Let $f\in \overline{E}_{+}\backslash \left\{\0\right\}$ and let $U\subset X$ be an open relatively compact nonempty set such that $f\ge\varepsilon\1_{U}$, for some $\varepsilon>0$. Then $\frac{1}{\varepsilon}f\in G^{\overline{E}}_{U}$, and so the latter set is nonempty. Hence, as was shown above, there is $g\in E$ such that $\0<g\le G^{\overline{E}}_{U}$, from where $\0<\varepsilon g\le f$. Since $f$ was chosen arbitrarily, we conclude that $E$ is order dense in $\overline{E}$.
\end{proof}

Note that in particular, a dense regular sublattice of $\Co\left(X\right)$ has to be weakly Urysohn. Consider examples of a dense sublattice that is not regular, and a dense regular sublattice that is not Urysohn.

\begin{example}\label{per}
Let $E$ be the set of all $f\in\Co\left(\R\right)$ which have an integer period, i.e. there is $n\in\N$ such that $f\left(x+n\right)=f\left(x\right)$, for all $x\in\R$. It is clear that if $f\in E$, then $\left|f\right|\in E$ and $\alpha f\in E$, for every $\alpha\in \R$. If $f,g\in E$, there are $n,m\in\N$ such that $f\left(x+n\right)=f\left(x\right)$ and $g\left(x+m\right)=g\left(x\right)$, for all $x\in\R$. Then, we also have $f\left(x+mn\right)=f\left(x\right)$ and $g\left(x+mn\right)=g\left(x\right)$, for all $x\in\R$, and so $f+g\in E$. Hence, $E$ is a sublattice of $\Co\left(\R\right)$, and it is easy to see that it strictly separates points of $\R$. Hence, $\overline{E}=\Co\left(\R\right)$ is trivially regular. If $E$ were regular, it would have to be weakly Urysohn, which it is not as $E_{\R\backslash\left(-1,1\right)}=\left\{\0\right\}$. Thus, $E$ is not regular.
\qed\end{example}

\begin{example}\label{perd}
Let $E$ be the set of all $f\in\Co\left(\R\right)$ for which there is $n\in\N$ such that $f\left(m\right)=f\left(0\right)$, whenever $m\in\Z$ with $\left|m\right|\ge n$. It is easy to see that $E$ is a dense order dense sublattice of $\Co\left(\R\right)$. However, if $U=\left(-1,1\right)$ and $f\in E$ vanishes outside of $U$, then $f\left(0\right)=0$, and so $E$ is not Urysohn.
\qed\end{example}

Let us focus now on dense sublattices of $\Co_{b}\left(X\right)$.

\begin{proposition}\label{dense}Let $E$ be a dense sublattice of $\Co_{b}\left(X\right)$. Then:
\item[(i)] $E$ is an order dense Urysohn sublattice.
\item[(ii)] If $H$ is a closed ideal in $\Co_{b}\left(X\right)$, then $\overline{E\cap H}=H$.
\item[(iii)] $E_{A}$ is a band in $E$ if and only if $\overline{A}$ is regularly closed; if $X$ is normal, then $E_{A}$ is a projection band in $E$ if and only if $\overline{A}$ is clopen.
\item[(iv)] If $H$ is a band / projection band in $E$, there is a (unique) regularly closed / clopen $A$ such that $\overline{H}=\Co_{b}\left(X\right)_{A}$ and $H=E_{A}$.
\end{proposition}
\begin{proof}
(i): follows immediately from Corollary \ref{sul}.

(ii): For this claim we may assume that $X$ is compact, so that $\Co_{b}\left(X\right)=\Co\left(X\right)$. Moreover, from Corollary \ref{id}, there is a closed $A\subset X$ such that $H=\Co\left(X\right)_{A}$. But then $E_{A}=E\cap H$ is dense in $H=\Co\left(X\right)_{A}$ due to Proposition \ref{sd}.\medskip

(iii): In the light of (i) and parts (i) and (ii) of Proposition \ref{ur} we only need to prove sufficiency in the second claim. As was mentioned earlier, every dense sublattice of $\Co_{b}\left(X\right)$ contains a strong unit, and so there is $e\in E_{+}$ such that $M_{e}$ is an automorphism of $\Co_{b}\left(X\right)$. Note that this automorphism also preserves zero-sets of functions, and so by replacing $E$ with $M_{e}^{-1}E$ without loss of generality we may assume that $\1\in E$. Then, from the Sublattice Urysohn lemma $E$ contains all indicators including $\1_{\overline{A}}$ and $\1_{X\backslash\overline{A}}$. If $f\in E$ with $\|f\|\le 1$, then $f\in E_{\overline{A}}+E_{X\backslash\overline{A}}$ by virtue of
$$f=f_{+}\wedge \1_{\overline{A}}-f_{-}\wedge \1_{\overline{A}} + f_{+}\wedge\1_{X\backslash\overline{A}} - f_{+}\wedge\1_{X\backslash\overline{A}}.$$

(iv): From Proposition \ref{dis} we have $H=H^{dd}\cap E$. Since $H^{dd}$ is a band in $\Co\left(X\right)$, from Corollary \ref{bpb}, there is  a regularly closed $A\subset X$ such that $H^{dd}=\Co\left(X\right)_{A}$. Hence, $H=\Co\left(X\right)_{A}\cap E=E_{A}$, and from (ii) we have $\overline{H}=\Co_{b}\left(X\right)_{A}$. If on top of that $H$ is a projection band in $E$, from part (ii) of Proposition \ref{ur}, $A$ is clopen.
\end{proof}

It follows from combining part (i) of Proposition \ref{dense} with parts (iii) and (iv) of Proposition \ref{ur} that $X$ is totally / extremally disconnect provided that $\Co_{b}\left(X\right)$ has a dense PPP / PP sublattice. Conversely, if $X$ is compact and totally disconnected, then the lattice of simple continuous functions strictly separates points, and so it is dense; it is not hard to verify that this sublattice has PPP. Hence, we arrive at the following characterization of totally disconnected compact spaces.

\begin{corollary}
If $X$ is compact, then it is totally disconnected if and only if $\Co\left(X\right)$ contains a dense sublattice with PPP.
\end{corollary}

It follows from Corollary \ref{nreg} that if $E$ is a regular sublattice of $\Co_{b}\left(X\right)$, then $\overline{E}$ is also regular. It turns out that the converse to this fact is partially true.

\begin{proposition}\label{rden}
If $E$ is sublattice of $\Co_{b}\left(X\right)$ that contains a strong unit of $\Co_{b}\left(X\right)$, then $E$ is order dense in $\overline{E}$. In particular, $E$ is regular if and only if  $\overline{E}$ is.
\end{proposition}
\begin{proof}
As in the proof of Proposition \ref{dense}, we may assume that $X$ is compact and $\1\in E$. Obviously, $E$ vanishes nowhere on $X$ and if $\delta_{x}^{E}=\alpha\delta_{y}^{E}$, for some $x,y\in X$, then $\alpha=\alpha\delta_{y}^{E}\left(\1\right)=\delta_{x}^{E}\left(\1\right)=1$, and so $E$ has only simple constraints. Hence, from Proposition \ref{quot}, there is a continuous surjection $\varphi$ from $X$ onto a compact space $Y$, such that $C_{\varphi}$ is an isometric isomorphism from $\Co\left(Y\right)$ onto $\overline{E}$. Then, $C_{\varphi}^{-1}E$ is a dense sublattice of $\Co\left(Y\right)$, and so it is order dense, according to part (i) of Proposition \ref{dense}. Since $C_{\varphi}$ facilitates an isomorphism between $\overline{E}$ and $\Co\left(Y\right)$, it follows that $E$ is order dense in $\overline{E}$. The second claim follows from Remark \ref{tran}.
\end{proof}

Let us show that without the assumption that $E$ contains a strong unit of $\Co_{b}\left(X\right)$ the preceding proposition is false.

\begin{example}\label{prrd}
Let $E$ be a sublattice of $c_{0}=\Co_{0}\left(\N\right)$ that consists of all $f:\N\to\R$ for which there is $n\in\N$ such that $f\left(m\right)=\frac{f\left(1\right)}{m}$, for every $m\ge n$. It is easy to check that $E$ is indeed a vector lattice, which even has a strong unit $e$, defined by $e\left(m\right)=\frac{1}{m}$, for $m\in\N$. Moreover, since $E$ strictly separates points of $\N$ it is dense in $\Co_{0}\left(\N\right)$. Hence, $\overline{E}$ is regular in $\Co_{b}\left(\N\right)$. On the other hand, $E$ is also dense in $\Co\left(\N\right)$, but not weakly Urysohn (as $E_{\N\backslash\left\{1\right\}}=\left\{\0\right\}$), and so not regular, according to the comment after Theorem \ref{sublo}.
\qed\end{example}

We conclude the section with another density result. For a sublattice $E$ of $\Co\left(X\right)$ and $A_{0},...,A_{n}\subset X$ let $E_{A_{0},...,A_{n}}$ be the set of all elements of $E_{A_{0}}$, which are constant on $A_{k}$, for every $k\in\overline{1,n}$. It is clear that $\Co_{b}\left(X\right)_{A_{0},...,A_{n}}$ has only simple constraints, and so from Proposition \ref{quot}, there is a continuous surjection $\varphi:\beta X\to Y$, where $Y$ is compact, such that $C_{\varphi}$ is an isometric homomorphism from $\Co\left(Y\right)_{\varphi\left(A_{0}\right)}$ onto $\Co\left(X\right)_{A_{0},...,A_{n}}$. Note that $\varphi\left(A_{0}\right)$ is either empty, or a singleton.

\begin{lemma}\label{dens}Let $E$ be a dense sublattice of $\Co_{b}\left(X\right)$ that contains $\1$. Then $E_{A_{0},...,A_{n}}$ is dense in $\Co_{b}\left(X\right)_{A_{0},...,A_{n}}$, for every $A_{0},...,A_{n}\subset X$. In particular, in the notations from above, $C_{\varphi}^{-1}E_{A_{0},...,A_{n}}$ is dense in $\Co\left(Y\right)_{\varphi\left(A_{0}\right)}$.
\end{lemma}
\begin{proof}
As was explained in Example \ref{sch}, we may assume that $X$ is compact. Furthermore, we can assume that all $A_{k}$ are closed and disjoint. Indeed, replacement of $A_{k}$ with $\overline{A_{k}}$ does not affect the sublattices in question, and if $A_{i}\cap A_{j}\ne\varnothing$, elements of these sublattices have to be constant on $A_{i}\cup A_{j}$, and so we can replace $A_{i}$ and $ A_{j}$ with $A_{i}\cup A_{j}$. The space $Y$ is the result of collapsing each of $A_{k}$ into a point (say $y_{k}$). From part (ii) of Proposition \ref{dense}, it is enough to show that $C_{\varphi}^{-1}E$ strictly separates points of $Y$.

For distinct $y,z\in Y$ one or both of them may belong to $\left\{y_{0},...,y_{n}\right\}$. Let $Z=\left\{z\right\}\cup \left\{y_{0},...,y_{n}\right\}\backslash \left\{y\right\}$. Then the sets $\varphi^{-1}\left(y\right)$ and $\varphi^{-1}\left(Z\right)$ are closed and disjoint, and so from the sublattice Urysohn lemma there is $f\in E$ which is equal $1$ on $\varphi^{-1}\left(y\right)$ and vanishes on $\varphi^{-1}\left(Z\right)$. Then, $f\in \Co\left(X\right)_{A_{0},...,A_{n}}$ and for $g=C_{\varphi}^{-1}f$ we have $g\left(y\right)=1$ and $g\left(z\right)=0$. Since $y,z$ were chosen arbitrarily, the result follows.
\end{proof}

\begin{remark}
It is possible prove a similar result for infinite number of $A_{k}$. However, one has to take a special care for separation of these sets. In particular, if there is a collection of sets $\left\{A_{i}\right\}_{i\in I}$ such that for every $i\in I$ there is $f\in \left[\1_{A_{i}}, \1_{X\backslash\bigcup\limits_{j\ne i} A_{j}}\right]$, and $i_{0}\in I$, then the set of elements of $E$ that vanish on $A_{i_{0}}$ and are constant on every $A_{i}$ is dense in the set of all bounded continuous functions on $X$ with this property.
\qed\end{remark}

\section{Some classes of continuous maps between topological spaces.}\label{m}

In this section $X$ and $Y$ are Tychonoff spaces. We will consider several classes of continuous maps between $X$ and $Y$. Most of them appear in the literature under various (often conflicting) names and so we chose to gather the most relevant information about them here, rather than leave it in the references. Let us start with the maps that do not make ``large'' sets ``small''. Namely, we will call a continuous $\varphi:X\to Y$:
\begin{itemize}
\item \emph{quasi-open} if $\Int\varphi\left(U\right)\ne\varnothing$, for every open nonempty $U\subset X$;
\item \emph{almost open} if $\Int\overline{\varphi\left(U\right)}\ne\varnothing$, for every open nonempty $U\subset X$;
\item \emph{(strongly) skeletal} if it is almost (quasi-) open onto its image.
\end{itemize}

Obviously, every quasi-open map is almost open. It is easy to see that a quasi-open map is almost open, provided it is locally close, i.e. every $x\in X$ has a closed neighborhood $V$ such that $\left.\varphi\right|_{V}$ is a closed map. In particular, this is the case if $X$ is locally compact, because then every $x\in X$ has a compact neighborhood, and every continuous map on a compact space is closed. A restriction of an almost / quasi-open map to an open subset of $X$ is almost / quasi-open map. Similar implications also hold between skeletal and strongly skeletal maps. Various additional information about the introduced classes of maps see e.g. in \cite{bkm,burke,mr}. Consider an example of an almost open bijection which is not quasi-open.

\begin{example}\label{aowo}
Let $X=\Q\oplus \left(\R\backslash \Q\right)$ be the disjoint union of rationals and irrationals. Then the natural embedding of $X$ into $\R$ is an almost open bijection, which is not quasi-open.
\qed\end{example}

We will now present some equivalents to the definitions above.

\begin{proposition}\label{ao}The following conditions are equivalent:
\item[(i)] $\varphi$ is quasi-open;
\item[(ii)] Every open nonempty $U\subset X$ contains an open nonempty subset $W$ such that $\varphi\left(W\right)$ is open;
\item[(ii')] Every open nonempty $U\subset X$ contains an open subset $W$ such that $\overline{W}=U$ and $\varphi\left(W\right)$ is open;
\item[(iii)] Preimage under $\varphi$ of every dense set is dense.
\end{proposition}
\begin{proof}
(i)$\Rightarrow$(ii'): Let $U\subset X$ be open and nonempty and let $W=\varphi^{-1}\left(\Int\varphi\left(U\right)\right)$. Clearly, $\varphi\left(W\right)=\Int\varphi\left(U\right)$ is open, and so it is left to show that $W$ is dense in $U$. Indeed, if $V=U\backslash\overline{W}\ne\varnothing$, then $\varnothing\ne\Int\varphi\left(V\right)\subset \Int\varphi\left(U\right)$, but $V\cap\varphi^{-1}\left(\Int\varphi\left(U\right)\right)=\varnothing$. Contradiction.\medskip

(ii')$\Rightarrow$(ii) is trivial. (ii)$\Rightarrow$(iii): Let $A\subset Y$ be dense and let $U= X\backslash\overline{\varphi^{-1}\left(A\right)}$. If $U$ is nonempty, from (ii) there is a nonempty open $W\subset U$ such that $V=\varphi\left(W\right)$ is open. Then, $V\cap A=\varnothing$ which contradicts density of $A$.\medskip

(iii)$\Rightarrow$(i): If $\Int\varphi\left(U\right)=\varnothing$, for some open $U\subset X$, then $Y\backslash \varphi\left(U\right)$ is dense, and so $X\backslash U\supset \varphi^{-1}\left(Y\backslash \varphi\left(U\right)\right)$ is dense in $X$. Hence, $U=\varnothing$.
\end{proof}

\begin{proposition}\label{wo}The following conditions are equivalent:
\item[(i)] $\varphi$ is almost open;
\item[(ii)]  $\Int\varphi^{-1}\left(\overline{V}\right)\subset \overline{\varphi^{-1}\left(V\right)}$, for every open $V\subset Y$;
\item[(iii)] Preimage under $\varphi$ of every open dense set is dense;
\item[(iv)] Preimage under $\varphi$ of every nowhere dense set is nowhere dense;
\item[(iv')] Preimage under $\varphi$ of every closed nowhere dense set is nowhere dense.
\end{proposition}
\begin{proof}
First, (ii)$\Rightarrow$(iii) and (iv)$\Rightarrow$(i) are trivial; (iv)$\Leftrightarrow$(iv') is true since every nowhere dense set is contained in a closed nowhere dense set.\medskip

(i)$\Rightarrow$(ii): Let $V\subset Y$ be open and let $U=\Int\varphi^{-1}\left(\overline{V}\right)\backslash \overline{\varphi^{-1}\left(V\right)}$, which is open. Then $U\subset \varphi^{-1}\left(\overline{V}\right)\backslash \varphi^{-1}\left(V\right)$, and so $\varphi\left(U\right)\subset\overline{V}\backslash V$ is nowhere dense. Hence, $U=\varnothing$, and so $\Int\varphi^{-1}\left(\overline{V}\right)\subset \overline{\varphi^{-1}\left(V\right)}$.\medskip

(iii)$\Rightarrow$(iv): Let $A\subset Y$ be nowhere dense. Then $V=Y\backslash\overline{A}$ is an open dense set, and so $\varphi^{-1}\left(V\right)$ is open dense and disjoint from $\varphi^{-1}\left(A\right)$. Hence, the latter is nowhere dense.
\end{proof}

Let us also remark that $\varphi$ is strongly skeletal if and only if for every open nonempty $U\subset X$ there is an open nonempty $V\subset Y$ such that $\varphi^{-1}\left(V\right)\subset \varphi^{-1}\left(\varphi\left(U\right)\right)$, since the latter formula simply means that $V\cap\varphi\left(X\right)\subset \varphi\left(U\right)$. Similarly, $\varphi$ is skeletal if and only if for every open nonempty $U\subset X$ there is an open nonempty $V\subset Y$ such that $\varphi^{-1}\left(V\right)\subset \varphi^{-1}\left(\overline{\varphi\left(U\right)}\right)$. The first observation motivates the following definition.

We will say that $A\subset X$ is $\varphi$-\emph{saturated} if $A=\varphi^{-1}\left(\varphi\left(A\right)\right)$. It is easy to see that $A$ is $\varphi$-saturated if and only if $A=\varphi^{-1}\left(B\right)$, for some $B\subset Y$, from where the union of any collection of $\varphi$-saturated sets is $\varphi$-saturated, as well as the difference of two $\varphi$-saturated sets. It now follows that if $\varphi$ is strongly skeletal, then for every open nonempty $U \subset X$, there is an open nonempty $\varphi$-saturated $W\subset \varphi^{-1}\left(\varphi\left(U\right)\right)$.\medskip

Let us move on to the maps that do not make ``small'' sets ``large''. We will call $\varphi$:
\begin{itemize}
\item \emph{weakly injective} if every open nonempty $U\subset X$ contains an open nonempty $\varphi$-saturated subset;
\item \emph{almost injective} if the set of $x\in X$ such that $\left\{x\right\}$ is $\varphi$-saturated is dense in $X$;
\item \emph{irreducible} if there no closed $A\subsetneq X$ is such that $\varphi\left(X\right)\subset\overline{\varphi\left(A\right)}$.
\end{itemize}

Obviously, every injective map is weakly injective and almost injective, but Example \ref{aowo} shows that an almost open bijection is not necessarily irreducible. An example of a bijection from a locally compact space, which is not irreducible is the natural projection of $\left\{0\right\}\oplus \left(0,1\right]$ into $\left[0,1\right]$. Let us discuss weak injectivity in slightly more details.

\begin{proposition}\label{wi}If $\varphi$ is weakly injective, then:
\item[(i)] Every open nonempty $U\subset X$ contains an open $\varphi$-saturated subset $W$ which is dense in $U$.
\item[(ii)] $\varphi^{-1}\left(\varphi\left(A\right)\right)\backslash \overline{A}$ is nowhere dense, for every $A\subset X$.
\item[(iii)] If $A\subset X$ is nowhere dense, then $\Int_{\varphi\left(X\right)}\varphi\left(A\right)=\varnothing$.
\end{proposition}
\begin{proof}
(i): Let $W$ be the union of all open $\varphi$-saturated subsets of $U$. Clearly, $W$ is open and $\varphi$-saturated, and so we only need to show that $U\subset\overline{W}$. If this is not true, then $V=U\backslash\overline{W}$ is open and nonempty, and so it contains an nonempty open $\varphi$-saturated subset. That subset is therefore an open $\varphi$-saturated subset of $U$, and so it is contained in $W$. Contradiction.\medskip

(ii): According to (i), for every $A\subset X$ there is an open $\varphi$-saturated $W\subset X\backslash\overline{A}$, which is dense in $X\backslash \overline{A}$. Since $W$ is $\varphi$-saturated, the same is true about $X\backslash W$. The latter set contains $A$, and so it also contains $\varphi^{-1}\left(\varphi\left(A\right)\right)$. Hence, $\varphi^{-1}\left(\varphi\left(A\right)\right)\backslash \overline{A}$ is contained in $\left(X\backslash\overline{A}\right)\backslash W$, which is nowhere dense.\medskip

(iii): If $A\subset X$ is nowhere dense, but such that $V=\Int_{\varphi\left(X\right)}\varphi\left(A\right)\ne\varnothing$, then $\varphi^{-1}\left(V\right)\not\subset \overline{A}$, since the latter is nowhere dense, and so $\varphi^{-1}\left(\varphi\left(A\right)\right)\backslash \overline{A}$ is not nowhere dense. This contradicts (ii).
\end{proof}

Let us also remark that if $\varphi$ is irreducible, then it maps nowhere dense sets into nowhere dense subsets of $\varphi\left(X\right)$. Indeed, if $A\subset X$ is closed nowhere dense and such that $V=\Int_{\varphi\left(X\right)}\overline{\varphi\left(A\right)}\ne\varnothing$, then $B=\left(X\backslash \varphi^{-1}\left(V\right)\right)\cup A$ is a closed proper subset of $X$ whose image is dense in $\varphi\left(X\right)$.

Note that in the definition of an irreducible map one often also assumes that $\varphi$ is closed or surjective (see \cite[Appendix 4]{alek}, \cite{bot}, \cite[III.1]{mr}, \cite[6.5]{pw} and \cite[VIII.10]{whyburn}), but we will not add those conditions. It is easy to see that any topological embedding (i.e. a homeomorphism onto its image) is irreducible.

\begin{proposition}\label{irr}The following conditions are equivalent:
\item[(i)] $\varphi$ is irreducible;
\item[(ii)] For every open nonempty $U\subset X$ there is an open $V\subset Y$ such that $\varnothing\ne\varphi^{-1}\left(V\right)\subset U$;
\item[(ii')] For every open nonempty $U\subset X$ there is an open $V\subset Y$ such that $\varphi^{-1}\left(V\right)$ is a dense subset of $U$;
\item[(iii)] $\varphi$ is strongly skeletal and weakly injective;
\item[(iv)] $\varphi$ is strongly skeletal and $\varphi^{-1}\left(\varphi\left(A\right)\right)$ is not dense, for every closed $A\subsetneq X$.
\end{proposition}
\begin{proof}
(i)$\Rightarrow$(ii): Assume that $U\subset X$ is open nonempty and such that for every open $V\subset Y$ either $\varphi^{-1}\left(V\right)=\varnothing$, or $\varphi^{-1}\left(V\right)\not\subset U$. Let us show that $\varphi\left(X\backslash U\right)$ is dense in $\varphi\left(X\right)$. Indeed, if $V$ is an open subset of $Y$ such that $V\cap \varphi\left(X\right)\ne\varnothing$, then $\varphi^{-1}\left(V\right)\ne\varnothing$ and so $\varphi^{-1}\left(V\right)\backslash U\ne \varnothing$, from where $\varphi\left(X\backslash U\right)\cap V\ne\varnothing$.\medskip

(ii)$\Rightarrow$(ii') is done by a similar argument to the part (i) of Proposition \ref{wi}. (ii')$\Rightarrow$(ii)$\Rightarrow$(iii) is easy to see, and (iii)$\Rightarrow$(iv) follows from part (ii) of Proposition \ref{wi}.\medskip

(iv)$\Rightarrow$(i): Assume that $\varphi$ is not irreducible. Then, there is a closed $A\subsetneq X$ such that $\varphi\left(A\right)$ is dense in $\varphi\left(X\right)$. Since $\varphi$ is strongly skeletal, it is quasi-open onto $\varphi\left(X\right)$, and so from Proposition \ref{ao}, the pre-image $\varphi^{-1}\left(\varphi\left(A\right)\right)$ has to be dense. Contradiction.\end{proof}

The following result shows that irreducibility, weak injectivity and almost injectivity coincide if $X$ is metrizable compact.

\begin{proposition}\label{mirr}
\item[(i)] If $\varphi$ is closed, then $\varphi$ is irreducible if and only if $\varphi^{-1}\left(\varphi\left(A\right)\right)\ne X$, for every closed $A\subsetneq X$.
\item[(ii)] If $\varphi$ is closed and almost injective, then it is irreducible.
\item[(iii)] If $X$ is a metrizable Baire space (e.g. completely metrizable) and $\varphi$ is weakly injective, then it is almost injective.
\end{proposition}
\begin{proof}
(i): From Proposition \ref{irr} we only have prove sufficiency; moreover, it is enough to show that $\varphi$ is skeletal. Replacing $Y$ with $\varphi\left(X\right)$ we may assume that $\varphi$ is a closed surjection, and so we need to prove that $\varphi$ is almost open. Assume that there is an open nonempty $U\subset X$ such that $\varphi\left(U\right)$ is nowhere dense and let $A=X\backslash U$. Since $\varphi$ is closed, $\overline{\varphi\left(A\right)}=\varphi\left(A\right)$, and since $\varphi\left(U\right)$ is nowhere dense, $\varphi\left(U\right)\subset \overline{Y\backslash \varphi\left(U\right)}\subset \overline{\varphi\left(A\right)}=\varphi\left(A\right)$. Hence, $U\subset \varphi^{-1}\left(\varphi\left(A\right)\right)$, from where $\varphi^{-1}\left(\varphi\left(A\right)\right)=X$. Contradiction.\medskip

(ii): Let $Z$ be the set of all $x\in X$ for which $\varphi^{-1}\left(\varphi\left(x\right)\right)=\left\{x\right\}$. From our assumption, $Z$ is dense in $X$. Assume that $A\subset X$ is closed and such that $\varphi\left(X\right)\subset\overline{\varphi\left(A\right)}$. Then, as $\varphi$ is a closed map, $\varphi\left(X\right)\subset\varphi\left(A\right)$. Since for $x\in Z$ the condition $\varphi\left(x\right)\in \varphi\left(A\right)$ implies $x\in A$, it follows that $Z\subset A$, and so $A=X$.\medskip

(iii): Fix a metric on $X$ and let $A_{n}$ be the set  of all $x\in X$ such that the diameter of $\varphi^{-1}\left(\varphi\left(x\right)\right)$ is at least $\frac{1}{n}$. We need to prove that $X\backslash \bigcup \limits_{n\in\N}A_{n}$ is dense. To that end, since $X$ is a Baire space, it is enough to show that $A_{n}$ is nowhere dense, for each $n\in\N$. Assume that $\Int\overline{A_{n}}\ne\varnothing$, for some $n\in\N$, and take an open $U\subset \Int\overline{A_{n}}$ of diameter less than $\frac{1}{n}$. Since $\varphi$ is weakly injective, there is an open nonempty $\varphi$-saturated $W\subset U$. For every $x\in W$ we have that $\varphi^{-1}\left(\varphi\left(x\right)\right)\subset W\subset U$, and so the diameter of $\varphi^{-1}\left(\varphi\left(x\right)\right)$ is less than $\frac{1}{n}$. Hence $W\cap A_{n}=\varnothing$, which contradicts $W\subset U\subset \Int\overline{A_{n}}$.
\end{proof}

\section{Order continuity of composition operators}\label{c}

In this section $X$ and $Y$ are Tychonoff spaces and $\varphi:X\to Y$ is continuous. We will study how some of the properties of $\varphi$ discussed in the previous section are reflected on the properties of the corresponding composition operator. Let $E\subset \Co\left(Y\right)$ and $F\subset \Co\left(X\right)$ be sublattices such that $C_{\varphi}E\subset F$. It is easy to verify that if $A\subset X$ and $B\subset Y$, then $C_{\varphi}E_{B}\subset F_{\varphi^{-1}\left(B\right)}$ and $C_{\varphi}^{-1}F_{A}=E_{\varphi\left(A\right)}$. In particular, $\Ker C_{\varphi}=C_{\varphi}^{-1}F_{X}=E_{\varphi\left(X\right)}$, and so if $E$ is a weakly Urysohn sublattice, then $C_{\varphi}$ is an injection if and only if $\overline{\varphi\left(X\right)}=Y$. The following is a generalization of \cite[Theorem 3.4 and Proposition 3.6]{kv1}.

\begin{theorem}\label{com}
If $E\subset\Co\left(Y\right)$ is an order dense Urysohn sublattice, then:
\item[(i)] $C_{\varphi}E$ is weakly Urysohn if and only if $C_{\varphi}E$ is order dense and if and only if $\varphi$ is irreducible.
\item[(ii)] $C_{\varphi}E$ is an Urysohn sublattice if and only if $\varphi$ is a topological embedding.
\item[(iii)] $C_{\varphi}:E\to \Co\left(X\right)$ is order-continuous if and only if $\varphi$ is almost open and if and only if $C_{\varphi}E_{A}^{dd}\subset\left(C_{\varphi}E_{A}\right)^{dd}$, for every $A\subset Y$.
\item[(iv)] $C_{\varphi}E$ is regular in $\Co\left(X\right)$ if and only if $\varphi$ is skeletal.
\end{theorem}
\begin{proof}
(i): First recall that every order dense sublattice is weakly Urysohn. If $C_{\varphi}E$ is a weakly Urysohn sublattice, then for every open nonempty $U\subset X$ there is $f\in  E$ such that $\0<f\circ\varphi\in F_{X\backslash U}$. Let $V=Y\backslash f^{-1}\left(0\right)$, which is open; we have $\varnothing\ne \varphi^{-1}\left(V\right)\subset U$. Since $U$ was chosen arbitrarily, it follows from Proposition \ref{irr} that $\varphi$ is irreducible.

Assume that $\varphi$ is irreducible and let $U\subset X$ be open and nonempty. There is an open $V\subset Y$ such that $\varnothing\ne \varphi^{-1}\left(V\right)\subset U$. Since $E$ is order dense, there is $f\in \left(\0,\1_{V}\right]\cap E$. Then, $f\circ\varphi\in \left(\0,\1_{U}\right]$, and since $U$ was chosen arbitrarily, it follows from part (ii) of Proposition \ref{sub} that $C_{\varphi}E$ is order dense.\medskip

(ii): Sufficiency is easy to see (essentially it is enough to check that restrictions of the elements of an Urysohn sublattice form an Urysohn sublattice). Necessity: Fix $x\in X$ and a closed $A\subset X$ which does not contain $x$. Since $C_{\varphi}E$ is an Urysohn sublattice, there is $f\in E$ such that $C_{\varphi}f$ vanishes on $A$, but not at $x$. In other words, $C_{\varphi}f\in \Co\left(X\right)_{A}\backslash \Co\left(X\right)_{\left\{x\right\}}$, or equivalently, $f\in E_{\varphi\left(A\right)}\backslash E_{\left\{\varphi\left(x\right)\right\}}$. Since $E$ is an Urysohn sublattice of $\Co\left(Y\right)$, existence of $f$ is equivalent to $\varphi\left(x\right)\not\in \overline{\varphi\left(A\right)}$. Substituting a singleton instead of $A$ yields injectivity of $\varphi$. Running $x$ through all points of $X\backslash A$ shows that $\varphi\left(A\right)$ is closed in $\varphi\left(X\right)$. Hence, $\varphi$ is an injection, which is a closed map onto its image. Thus, it is a topological embedding.\medskip

(iii): Assume that $\varphi$ is almost open. Let $f_{\gamma}\downarrow_{E} \0$. Since $E$ is regular, we get $f_{\gamma}\downarrow_{\Co\left(Y\right)} \0$. Assume that there is $g\in \Co\left(X\right)$ such that $\0< g \le f_{\gamma}\circ\varphi$, for every $\gamma\in\Gamma$. Then, there is an open $U\subset X$ and $\varepsilon>0$ such that $g\ge \varepsilon \1_{U}$, from where  $f_{\gamma}\ge\varepsilon \1_{\varphi\left(U\right)}$, for every $\gamma$; from continuity, $f_{\gamma}\ge\varepsilon \1_{\overline{\varphi\left(U\right)}}\ge\varepsilon\1_{\Int\overline{\varphi\left(U\right)}}>\0 $. This contradicts  $f_{\gamma}\downarrow_{\Co\left(Y\right)} \0$, due to Proposition \ref{inf}. Thus,  $f_{\gamma}\circ\varphi \downarrow_{\Co\left(X\right)} \0$, and so $C_{\varphi}$ is order-continuous.\medskip

Since $C_{\varphi}$ is a homomorphism, from (already proven) implication (i)$\Rightarrow$(v) in Theorem \ref{hoc}, order continuity of $C_{\varphi}$ implies $C_{\varphi}G^{dd}\subset\left(C_{\varphi}G\right)^{dd}$, for every $G\subset E$, including $G=E_{A}$, for $A\subset Y$.\medskip

Assume that $C_{\varphi}E_{A}^{dd}\subset\left(C_{\varphi}E_{A}\right)^{dd}$, for every $A\subset Y$. Let $U\subset X$ be open and nonempty. From the observations before the theorem we have $C_{\varphi}E_{\overline{\varphi\left(U\right)}}\subset \Co\left(X\right)_{\varphi^{-1}\left(\overline{\varphi\left(U\right)}\right)}\subset \Co\left(X\right)_{U}$. Since the latter is a band in $\Co\left(X\right)$, we get $$C_{\varphi}E_{\Int \overline{\varphi\left(U\right)}}=C_{\varphi}E_{\overline{\varphi\left(U\right)}}^{dd}\subset\left(C_{\varphi}E_{\overline{\varphi\left(U\right)}}\right)^{dd}\subset \Co\left(X\right)_{U}^{dd}=\Co\left(X\right)_{U},$$ from where $E_{\Int \overline{\varphi\left(U\right)}}\subset C_{\varphi}^{-1}\Co\left(X\right)_{U}=E_{\varphi\left(U\right)}$. Since $E$ is an Urysohn sublattice it follows that $\overline{\varphi\left(U\right)}\subset \overline{\Int \overline{\varphi\left(U\right)}}$, and so $\Int \overline{\varphi\left(U\right)}\ne\varnothing$. As $U$ was chosen arbitrarily, we conclude that $\varphi$ is almost open.\medskip

(iv): Let $Z=\varphi\left(X\right)\subset Y$, let $\psi$ be the inclusion map from $Z$ into $Y$ and let $\varphi'$ be $\varphi$ viewed as a map from $X$ onto $Z$. Since $\psi$ is a topological embedding, it is irreducible, and so from (i) and (ii) $C_{\psi}E$ is an order dense Urysohn sublattice of $\Co\left(Z\right)$. Hence, from (iii) almost openness of $\varphi'$ is equivalent to order continuity of $C_{\varphi'}$. Since $\varphi'$ is a surjection, $C_{\varphi'}$ is an injection, and so from (already proven) equivalency (i)$\Leftrightarrow$(iii) in Theorem \ref{hoc}, order continuity of $C_{\varphi'}$ is equivalent to regularity of $C_{\varphi'}C_{\psi}E=C_{\varphi}E$ in $\Co\left(X\right)$.
\end{proof}

\begin{remark}
It is not hard to prove that a multiplication operator is always order continuous. Also, similarly to part (iii) one can show that if $g\in\Co\left(X\right)$, then $M_{g}C_{\varphi}$ is order continuous if and only if $\left.\varphi\right|_{X\backslash g^{-1}\left(0\right)}$ is almost open.
\qed\end{remark}

Let us use the obtained information to complete the proof of Theorem \ref{hoc}. In order to do so we will use the following fact which says that any homomorphism is locally a composition operator.

\begin{theorem}[Krein-Kakutani representation theorem] Let $T:F\to E$ be a homomorphism between Archimedean vector lattices, such that $Tf=e$, where $f$ and $e$ are strong units of $F$ and $E$ respectively. Then, there exist
\begin{itemize}
\item compact spaces $K$ and $L$ and a continuous map $\varphi:K\to L$,
\item dense sublattices $H$ and $G$ of $\Co\left(K\right)$ and $\Co\left(L\right)$, respectively,
\item isomorphisms $S:G\to F$ and $R:H\to E$,
\end{itemize}
such that $f=S\1_{L}$, $e=R\1_{K}$ and $T=R\left.C_{\varphi}\right|_{G}S^{-1}$.
\end{theorem}
\begin{proof}
Existence of $K$, $L$, $H$, $G$, $S$ and $R$ follows from the usual Kakutani representation theorem (see e.g. \cite[Theorem 2.1.3]{mn}). Note that $S$ is an isometry with respect to $\|\cdot\|_{f}$ and $\|\cdot\|$, while $R$ is an isometry with respect to $\|\cdot\|_{e}$ and $\|\cdot\|$. Since $T$ is a homomorphism and $Tf=e$, it follows that $T\left[-f,f\right]\subset \left[-e,e\right]$, and so $T$ is continuous with respect to $\|\cdot\|_{f}$ and $\|\cdot\|_{e}$. Hence, $R^{-1}TS$ is a continuous homomorphism from a dense sublattice of $\Co\left(L\right)$ into $\Co\left(K\right)$. Therefore it admits an extension to a continuous homomorphism from $\Co\left(L\right)$ into $\Co\left(K\right)$, which maps $\1_{L}$ into $\1_{K}$, and so is a composition operator (see e.g. \cite[Theorem 3.2.12]{mn}).\end{proof}

\begin{proof}[Proof of (v')$\Rightarrow$(i) in Theorem \ref{hoc}]
Fix $f\in F_{+}$ and let $e=Tf$. Apply the Krein-Kakutani representation theorem to $\left.T\right|_{F_{f}}:F_{f}\to E_{e}$ and produce $K$, $L$, $\varphi$, $H$, $G$, $R$ and $S$. Note that $G$ is dense in $\Co\left(L\right)$, and so from part (i) of Proposition \ref{dense} it is an order dense Urysohn sublattice.

Let $A\subset L$, and let $B=SG_{A}\subset F_{f}$. Since $S$ is an isomorphism, $SG_{A}^{dd}=\left(SG_{A}\right)^{dd}_{F_{f}}=B^{dd}_{F_{f}}$. Since $F_{f}$ and $E_{e}$ are regular in $F$ and $E$ respectively, from the condition (v) and (already proven) implication (i)$\Rightarrow$(v) of Corollary \ref{reg}, we have
$$TB^{dd}_{F_{f}}=T\left(B^{dd}\cap F_{f}\right)\subset T\left(B^{dd}\right)  \cap TF_{f}\subset \left(TB\right)^{dd}\cap E_{e}= \left(TB\right)^{dd}_{E_{e}}.$$
Since $R$ is an isomorphism and $H$ is regular in $\Co\left(K\right)$ we have
$$C_{\varphi}G_{A}^{dd}=R^{-1}TSG_{A}^{dd}=R^{-1}TB^{dd}_{F_{f}}\subset R^{-1}\left(TB\right)^{dd}_{E_{e}}=\left(R^{-1}TB\right)^{dd}_{H}=\left(C_{\varphi}G_{A}\right)^{dd}_{\Co\left(K\right)}\cap H.$$
As $G$ is an order dense Urysohn sublattice, and $A$ was chosen arbitrarily, $C_{\varphi}$ is order-continuous, by virtue of part (iii) of Theorem \ref{com}. Hence, $\left.T\right|_{F_{f}}=R\left.C_{\varphi}\right|_{G}S^{-1}$ is order continuous. Since $f\in F_{+}$ was chosen arbitrarily, from Corollary \ref{oc}, $T$ is order continuous.
\end{proof}

Let us finalize the proof of Theorem \ref{hoc}.

\begin{proof}[Proof of (vi')$\Rightarrow$(i) in Theorem \ref{hoc}]
First, note that if the set $G=\left\{h\in F_{+},~\dim TF_{h}=\8\right\}$ is nonempty, then it is majorizing, since if $g\in G$ and $f\in F_{+}$, then $f+g\in G$ and $f\le f+g$. Also note that if $F$ has the $\sigma$-property and $\dim TF=\8$, then $G\ne\varnothing$. Indeed, since $\dim TF=\8$ there are $\left\{f_{n}\right\}_{n\in\N}$ such that $\left\{Tf_{n}\right\}_{n\in\N}$ are linearly independent, and from $\sigma$-property there is $f\in F_{+}$ such that $\left\{f_{n}\right\}_{n\in\N}\subset F_{f}$, and so $f\in G$.

Hence, similarly to (v)$\Rightarrow$(i), we can reduce the proof to the case when $F$ and $E$ are dense sublattices of $\Co\left(L\right)$ and $\Co\left(K\right)$, respectively, where $K$ and $L$ are compact, such that $\1_{L}\in F$ and $\1_{K}\in E$, and $T=C_{\varphi}$, for some $\varphi:K\to L$. It is given that $\dim C_{\varphi}F=\8$ and $C_{\varphi}H$ is regular in $E$ (and so in $\Co\left(K\right)$), for every order dense sublattice $H$ of $F$, and we need to show that $\varphi$ is almost open (note that order dense sublattices of $F$ correspond to order dense sublattices of $F_{f}$ according to Remark \ref{menag}).

The condition $\dim C_{\varphi}F=\8$ implies that $\varphi\left(K\right)$ is an infinite subset of the compact space $L$, and so it has an accumulation point, say $y$. Then $\left\{y\right\}$ is a nowhere dense in $\varphi\left(K\right)$, and so combining part (iv) of Theorem \ref{com} with Proposition \ref{wo} yields that $\varphi^{-1}\left(y\right)$ is nowhere dense. Assume that $U\subset K$ is open and such that $\varphi\left(U\right)$ is nowhere dense. Let $x\in U\backslash \varphi^{-1}\left(y\right)$.

Let $V$ be an open subset of $L$ such that $\varphi\left(x\right)\in V$ and $y\in W=L\backslash\overline{V}$. Then, $y$ is an accumulation point of $\varphi\left(K\right)\cap W$. Define $M=L\slash\left(\overline{\varphi\left(U\right)}\cap\overline{V}\cup\left\{y\right\}\right)$ and let $\psi: L\to M$ be the corresponding quotient map. In other words, $\psi$ collapses a nowhere dense closed set $\overline{\varphi\left(U\right)}\cap\overline{V}\cup\left\{y\right\}$ into a single point (call it $o$), which is an accumulation point of $\psi\left(\varphi\left(K\right)\cap W\right)$ in $M$. Hence, $\psi\circ\varphi$ sends an open set $U\cap\varphi^{-1}\left(V\right)$, which contains $x$, into $\left\{o\right\}$ which is nowhere dense in $\psi\left(\varphi\left(K\right)\right)$, and so it is not a skeletal map. At the same time since $\psi$ only collapses a closed nowhere dense set, according to Proposition \ref{irr}, $\psi$ is irreducible and $M$ is compact Hausdorff.

From Lemma \ref{dens}, $H=C_{\psi}^{-1}F$ is a dense sublattice of $\Co\left(M\right)$ that contains $\1$. Since $\psi$ is irreducible, from part (i) of Theorem \ref{com}, $C_{\psi}H$ is order dense in $\Co\left(L\right)$, and therefore in $F$. At the same time $\psi\circ\varphi$ is not skeletal, from where $C_{\varphi}C_{\psi}H=C_{\psi\circ\varphi}H$ is not regular in $\Co\left(K\right)$, due to part (iv) of Theorem \ref{com}. Thus, $C_{\varphi}$ sends an order dense sublattice $C_{\psi}H$ of $F$ into a non-regular sublattice of $\Co\left(K\right)$. Contradiction.
\end{proof}

Let us conclude the article with some additional details on the topic raised in Remark \ref{eqr}. Recall that if $X$ is compact, then there is a bijective correspondence between closed sublattices of $\Co\left(X\right)$ that contain $\1$ (and are also subalgebras), continuous surjections from $X$ onto compact spaces and closed equivalence relations on $X$. Theorem \ref{com} allows to characterize order density and regularity of such sublattices in terms of the surjections. Since a surjection $\varphi:X\to Y$ corresponds to the equivalence relation $x\sim y\Leftrightarrow \varphi\left(x\right)=\varphi\left(y\right)$, we can translate those characterizations into the language of the equivalence relations (note that propositions \ref{ao}, \ref{wo}, \ref{irr} and \ref{mirr} can be used to produce more equivalent characterizations).

\begin{proposition}\label{eqq}
Let $X$ be compact and let $E$ be a closed sublattice of $\Co\left(X\right)$ that contains $\1$ and let $\sim$ be the corresponding equivalence relation. Then:
\item[(i)] $E$ is regular if and only if every open nonempty $U\subset X$ contains $A\subset U$ such that the union of all classes of $\sim$ that intersect $A$ is open.
\item[(ii)] $E$ is order dense if and only if for every closed $A\ne X$ the union of all classes of $\sim$ that intersect $A$ is not the whole $X$.
\end{proposition}

Let $E$ and $F$ be closed sublattices of $\Co\left(X\right)$ that contain $\1$ and let $\sim$ and $\approx$ be the corresponding equivalence relations on $X$. One can show that $E\cap F$ corresponds to the smallest closed equivalence relation $\sim\vee\approx$ on $X$ that contains $\sim\cup\approx$, whereas $\sim\cap\approx$ corresponds to the smallest closed sublattice that contains both $E$ and $F$. Note that if $E$ and $F$ correspond to continuous surjections $\varphi:X\to Y$ and $\psi:X\to Z$, then $\sim\cap\approx$ also corresponds to the diagonal map $\varphi\times\psi:X\to Y\times Z$ (viewed as a surjection onto its image). Let us consider two examples that show that regularity of sublattices is not well behaved with respect to intersections and ``unions'', and order density is not well behaved with respect to intersections.

\begin{example}\label{inte}
Let $Y$ be the planar segment that joins $\left(0,0\right)$ with $\left(-1,-1\right)$ and let $X=Y\cup \left[0,1\right]^{2}$. Let $\sim$ and $\approx$ be the equivalence relations on $X$ that identifies points on $\left[0,1\right]^{2}$ along the vertical and horizontal segments respectively. It is easy to see that $\sim$ and $\approx$ satisfy the condition (i) in Proposition \ref{eqq}, and so the corresponding sublattices $E$ and $F$ are regular. It is easy to check that $\sim\vee\approx$ identifies all points on $\left[0,1\right]^{2}$. The criterion in part (i) of Proposition \ref{eqq} fails for any open subset of $\left[0,1\right]^{2}$, and so $E\cap F$ is not regular.
\qed\end{example}\smallskip

\begin{example}\label{inter}
Let $X$, $Y$, $\sim$ and $\approx$ be as in the previous example, and let $\sim'$ and $\approx'$ be extensions of $\sim$ and $\approx$, respectively, that also identify points $\left(t,t\right)$ and $\left(-t,-t\right)$, for every $t\in\left(0,1\right]$. It is easy to check that these are closed equivalence relations, which satisfy the condition (i) in Proposition \ref{eqq}, and so the corresponding sublattices $E'$ and $F'$ are regular. Let $H$ be the sublattice generated by $E$ and $F$. We will show that $\overline{H}$ is not regular, from where and Proposition \ref{rden} it will follow that $H$ is not regular. The equivalence relation that corresponds to $\overline{H}$ is $\sim'\cap\approx'$, which only identifies $\left(t,t\right)$ and $\left(-t,-t\right)$, for every $t\in\left(0,1\right]$. The corresponding surjection from $X$ onto $\left[0,1\right]^{2}$ sends an open set $Y\backslash\left\{\left(0,0\right)\right\}$ into a nowhere dense diagonal of the square. Therefore, this map is not almost open, and so $\overline{H}$ is not regular.
\qed\end{example}\smallskip

\begin{example}\label{intero}
Let $X=\left\{0,1\right\}^{\N}$ be the Cantor space, let $\varphi:X\to\left[0,1\right]$ be the ``evaluation of binary expressions'', i.e. $\varphi\left(a_{n}\right)_{n\in\N}=\sum\limits_{n\in\N}\frac{a_{n}}{2^{n}}$. This map identifies elements of the form $\left(A_{n},1,0,0,...\right)$ and $\left(A_{n},0,1,1,...\right)$, for $A_{n}\in \left\{0,1\right\}^{n}$. Let $X_{0}$ and $X_{1}$ be the set of elements of the first and second type, respectively. It is easy to see that both $X_{0}$ and $X_{1}$ as well $Y=X\backslash\left(X_{0}\cup X_{1}\right)$ are dense in $X$. Since $X$ is compact, it follows from part (ii) of Proposition \ref{mirr} that $\varphi$ is irreducible.

Let $\psi,\theta:X\to X$ be the transformations that interchange $2n-1$-th with $2n$-th and $2n$-th with $2n+1$-th coordinates, respectively, for every $n\in\N$. Since these maps are homeomorphisms, it follows that $\varphi\circ\psi$ and $\varphi\circ\theta$ are irreducible. Let $\sim$, $\approx$ and $\simeq$ be the equivalence relations generated by $\varphi$, $\varphi\circ\psi$ and $\varphi\circ\theta$, respectively, and let $\cong=\sim\vee\approx\vee\simeq$ ($\cong$ is the supremum of $\sim$, $\approx$ and $\simeq$). We will show that $\cong$ identifies all elements of $X$. This will imply that the intersection of the corresponding order dense sublattices of $\Co\left(X\right)$ contains only constants.

For $n\in\N_{0}=\N\cup\left\{0\right\}$ let $A_{2n}=\left(a_{1},...,a_{2n}\right)\in \left\{0,1\right\}^{2n}$, and let\linebreak $B_{2n}=\left(a_{2},a_{1},a_{4},a_{3},...,a_{2n},a_{2n-1}\right)$. We have $\psi\left(A_{2n},1,0,0,...\right)=\left(B_{2n},0,1,0,...\right)\sim \left(B_{2n},0,0,1,...\right)=\psi\left(A_{2n},0,0,1,...\right)$, from where $\left(A_{2n},1,0,0,...\right)\approx\left(A_{2n},0,0,1,...\right)$. Since $\left(A_{2n},0,0,1,...\right)\sim \left(A_{2n},0,1,0,...\right)$, it follows that $\left(A_{2n},1,0,0,...\right)\cong \left(A_{2n},0,1,0,...\right)$, and similarly $\left(A_{2n},0,1,0,...\right)\cong \left(A_{2n},1,1,0,...\right)$. Analogously, we also have\linebreak $\left(A_{2n+1},1,0,0,...\right)\cong \left(A_{2n+1},0,1,0,...\right)\cong \left(A_{2n+1},1,1,0,...\right)$, where $A_{2n+1}\in \left\{0,1\right\}^{2n+1}$.

Let $y_{0}=\left(0,...\right)$, $y_{1}=\left(1,0,...\right)$ and $X'_{0}=X_{0}\backslash \left\{y_{0}\right\}$. Since the latter is dense in $X$, it is enough to show that $\cong$ identifies every element of $X'_{0}$ with $y_{1}$. Let $x\in X'_{0}$, and let $n$ be the position of the last non-zero coordinate of $x$. We will prove our claim by induction over $n$. If $n=1$, then $x=y_{1}$, and so $x\cong y_{1}$. Assume that the claim is proven for $n=1,...,m$ and let $n=m+1$. If $n=2k$, for some $k\in\N$, we have that $x$ is either of the form $\left(A_{2k-2},1,1,0,...\right)$, or $\left(A_{2k-2},0,1,0,...\right)$. In both cases $x\cong \left(A_{2k-2},1,0,0,...\right)\cong y_{1}$, by the assumption of induction. The case when $n=2k-1$ is done similarly.
\qed\end{example}

Note that we showed that the intersection of three closed order dense sublattices can fail to be order dense. This means that it is not true that the intersection of two closed order dense sublattices is order dense. However, we did not construct  an explicit counterexample to that claim.

\begin{question}
Is there a simpler and explicit example of two closed order dense sublattices of a Banach lattice, such that their intersection is not order dense? Can such intersection be non-regular?
\end{question}

\section{Acknowledgements}

The author wants to thank Vladimir Troitsky for many valuable discussions on the topic of this paper, and Taras Banakh who contributed an idea for Example \ref{inter} and the service \href{mathoverflow.com/}{MathOverflow} which made it possible.

\end{document}